\documentclass[oneside]{amsart}
\usepackage{geometry}[margins=1 in]
\usepackage{amssymb, amsmath, amsfonts, amsthm, graphics}
\usepackage[T1]{fontenc}
\usepackage{enumitem}
\usepackage{float}
\usepackage{tikz}
\usepackage{dynkin-diagrams}
\usepackage{tikz-cd}
\usepackage{hyperref}
\usepackage{cleveref}

\newcommand{\newword}[1]{\textbf{\emph{#1}}}

\newcommand{\into}{\hookrightarrow}
\newcommand{\onto}{\twoheadrightarrow}

\newcommand{\join}{\vee}
\newcommand{\bigmeet}{\bigwedge} 
\newcommand{\bigjoin}{\bigvee}

\newcommand{\QQ}{\mathbb{Q}}
\newcommand{\RR}{\mathbb{R}}

\newcommand{\ZZ}{\mathbb{Z}}

\newcommand{\cX}{\mathcal{X}}

\newcommand{\tA}{\widetilde{A}}
\newcommand{\tB}{\widetilde{B}}
\newcommand{\tC}{\widetilde{C}}
\newcommand{\tD}{\widetilde{D}}

\newcommand{\tS}{\widetilde{S}}
\newcommand{\tT}{\widetilde{T}}
\newcommand{\tX}{\widetilde{X}}
\newcommand{\te}{\widetilde{e}}

\newcommand{\Fix}{\text{Fix}}

\newtheorem{theorem}{Theorem}
\numberwithin{theorem}{section}

\newtheorem{proposition}[theorem]{Proposition}

\newtheorem{corollary}[theorem]{Corollary}
\newtheorem{cor}[theorem]{Corollary}

\newtheorem{lemma}[theorem]{Lemma}

\theoremstyle{definition}
\newtheorem{remark}[theorem]{Remark}

\newtheorem{definition}[theorem]{Definition}

\newtheorem{eg}[theorem]{Example}

\newcommand{\op}{\text{op}}

\renewcommand{\a}{\alpha}

\DeclareMathOperator{\bic}{Bic}

\newcommand{\pos}{\Phi^+}

\DeclareMathOperator{\tot}{Tot}

\DeclareMathOperator{\faces}{Faces}

\newcommand{\An}{\widetilde{A}_{n-1}}
\newcommand{\tAn}{\An}

\usepackage[backend=bibtex]{biblatex}
\title{Combinatorial Descriptions of Biclosed Sets in Affine Type}

\date{}

\begin{document}

\begin{abstract}Let $W$ be a Coxeter group and let $\Phi^+$ be the positive roots. 
A subset $B$ of $\Phi^+$ is called ``biclosed'' if, whenever we have roots $\alpha$, $\beta$ and $\gamma$ with $\gamma \in \RR_{>0} \alpha + \RR_{>0} \beta$, if $\alpha$ and $\beta \in B$ then $\gamma \in B$ and, if $\alpha$ and $\beta \not\in B$, then $\gamma \not\in B$.
The finite biclosed sets are the inversion sets of the elements of $W$, and the containment between finite inversion sets is the weak order on $W$.
Dyer suggested studying the poset of all biclosed subsets of $\Phi^+$, ordered by containment, and conjectured that it is a complete lattice.
As progress towards Dyer's conjecture, we classify all biclosed sets in the affine root systems.
We provide both a type uniform description, and concrete models in the classical types $\tA$, $\tB$, $\tC$, $\tD$.
We use our models to prove that biclosed sets form a complete lattice in types $\tA$ and $\tC$, and we classify which biclosed sets are separable and which are weakly separable.
\end{abstract}

\author{Grant T. Barkley and David E Speyer}
\maketitle

\section{Main result}
This is a paper about Coxeter groups, affine root systems, and the weak order.
In this section, we state our main results for the reader who is already familiar with these concepts.
We will give a more gentle introduction in \Cref{sec:NiceIntro} by explaining what our results say in type $\tA$. 
We define all of our terminology in \Cref{Background} and prove our main theorem in \Cref{sec:AffineModels}.
In \Cref{sec:separable}, we describe which of our biclosed sets are weakly separable.
In \Cref{sec:combinatorial}, we expound on what our results mean for types $\tB$, $\tC$ and $\tD$.

Let $\Phi$ be an affine root system with corresponding affine Coxeter group $W$ and let $\Phi^+$ be the positive roots.
We write $V$ for the vector space in which $\Phi$ lives.
A subset $B$ of $\Phi^+$ is called \newword{biclosed} if, for any $\alpha$, $\beta$, $\gamma \in \Phi^+$ with $\gamma \in \RR_{>0} \alpha + \RR_{>0} \beta$:
\begin{enumerate}
	\item If $\alpha$ and $\beta \in B$, then $\gamma \in B$ and
	\item if $\alpha$ and $\beta \not\in B$, then $\gamma \not\in B$.
\end{enumerate}

We pause to discuss some related notions. For $w \in W$, the set of inversions of $w$ is denoted $N(w)$ and is always biclosed.
If $\rho \in V^{\ast}$ is a vector which has $\langle \rho, \beta \rangle > 0$ for all $\beta \in \Phi^+$, then $N(w) = \{ \beta \in \Phi^+ : \langle w \rho, \beta \rangle < 0 \}$. 
More generally, for any linear functional $\theta \in V^{\ast}$ such that $\langle \theta, \beta \rangle \neq 0$ for all $\beta \in \Phi^+$, the set $\{ \beta \in \Phi^+ : \langle \theta, \beta \rangle < 0 \}$ is always biclosed; a biclosed set of this form is called \newword{separable}. Even more generally, let $\theta_1$, $\theta_2$, \ldots, $\theta_n$ be a basis of $V^{\ast}$, then the set  $\{ \beta \in \Phi^+ : \exists k,\  \langle \theta_1, \beta \rangle=\langle \theta_2, \beta \rangle=\cdots=\langle \theta_{k-1}, \beta \rangle=0 \ \text{and}\ \langle \theta_k, \beta \rangle < 0 \}$ is biclosed; a biclosed set of this form is called \newword{weakly separable}.
If $B$ is a finite set of roots then all four of these notions -- biclosed, weakly separable, separable and an inversion set -- coincide. 
In general, they do not. Biclosed sets seem to be the best replacement for inversion sets when considering infinite $B$.
For this reason, we have taken on the task of classifying biclosed sets in affine Coxeter groups, and explaining how they relate to these other notions. 

Let $W_0$ be the corresponding finite Coxeter group and let $V_0$ be the geometric representation of $W_0$.
The \newword{Coxeter fan} in $V_0^*$ is the fan defined by the reflecting hyperplanes of $W_0$; 
a \newword{face} of this fan is a cone $F$ of any dimension in this fan, including $\{ 0 \}$.
We write $\Phi_0$ for the $W_0$-root system, so there is a map $\pi: \Phi \to \Phi_0$ which equates roots that differ by a multiple of the imaginary root. 

Let $F$ be a face of the $W_0$-Coxeter fan. We write $\Phi_0^F$ for set of roots in $\Phi_0$ which are orthogonal to every vector in $F$, and we write $\Phi_F$ for the set of roots $\beta \in \Phi$ with $\pi(\beta) \in \Phi^F_0$. 
Let $W_F$ be the reflection group generated by the reflections over the roots $\beta$ in $\Phi_F$; we call $W_F$ a \newword{face subgroup} of $W$.
We can uniquely decompose $\Phi_F$ as a disjoint union of mutually orthogonal irreducible root systems, which we call the \newword{components} of $\Phi_F$.
Our biclosed sets will be denoted $B(F, \Phi', w)$, where $F$ is a face of the $W_0$-Coxeter fan,  $\Phi'$ is a union of components of $\Phi_F$ and $w$ is in $W_F$. 
(We permit the cases $\Phi' = \emptyset$ and $\Phi' = \Phi_F$.)

The biclosed set $B(F, \Phi', e)$ is given by
\[ B(F, \Phi', e) = \{ \beta \in \Phi^+ : \langle F, \beta \rangle < 0 \} \sqcup (\Phi' \cap \Phi^+). \]
The general biclosed set $B(F, \Phi', w)$ is the symmetric difference between $B(F, \Phi', e)$ and the inversion set of $w$, where we identify inversions of $w \in W_F$ with elements of $\Phi^+_F$ in a way we will spell out in \Cref{subsec:CommensurabilityClasses}.

\begin{theorem} \label{thm:Main}
	The biclosed sets of $\Phi^+$ are precisely the sets $B(F, \Phi', w)$ as above, and we have listed each biclosed set once.
\end{theorem}

This result was found independently by Matthew Dyer in an unreleased work \cite{Dyerpreprint}; as a result, this article will emphasize explicit combinatorial descriptions in types $\tA, \tB, \tC, \tD$. There has been much previous work on classifying various special biclosed sets. The family that has seen the most attention are the \newword{compatible} or \newword{biconvex} sets (see, e.g., \cite{Cellini1998,Ito1999}), which are the inversion sets of \newword{infinite reduced words} of a Coxeter group (see, e.g., \cite{Hohlweg2016,LP,Wang2018}). The results here extend this previous work.

We also summarize our results on separable and weakly separable sets.
\begin{theorem} \label{thm:SeparableMain}
	The biclosed set $B(F, \Phi', w)$ is separable if and only if either $F$ is either a maximal face of the $W_0$-Coxeter fan, or else $F = \{ 0 \}$; in the latter case, $B$ is either of the form $N(w)$ or $\Phi^+ \setminus N(w)$. 
	The biclosed set $B(F, \Phi', w)$ is weakly separable if and only if either $\Phi' = \emptyset$ or $\Phi' = \Phi_F$.
\end{theorem}

The outline of this article is as follows. In \Cref{sec:NiceIntro}, we give a concrete introduction to our results in the case of the affine symmetric group. In \Cref{Background}, we describe the necessary background to prove \Cref{thm:Main}. \Cref{sec:AffineModels} is devoted to the proof of \Cref{Parametrization}, which gives a uniform description of affine-type biclosed sets in terms of faces of the Coxeter fan, refining \Cref{thm:Main}. 
In \Cref{sec:separable}, we discuss relations to separable and weakly separable sets.
In  \Cref{sec:combinatorial}, we give explicit combinatorial models for biclosed sets in the classical affine Coxeter groups and prove the results in \Cref{sec:NiceIntro}. 
Our combinatorial models in \Cref{sec:combinatorial} can be mostly understood independently of Sections~\ref{sec:AffineModels} and~\ref{sec:separable}.

GTB was supported by NSF grant DMS-1854512 and as an REU student by NSF grant DMS-1600223. DES was supported by NSF grants DMS-1600223, DMS-1854225 and DMS-1855135. The authors would also like to thank Matthew Dyer for helpful communications.

The results in this paper were described in an extended abstract appearing in the proceedings of FPSAC 2022 \cite{FPSAC2022}, and some of our text, especially in Sections~\ref{sec:NiceIntro}, \ref{Background} and the early parts of \ref{sec:AffineModels}, has overlap with that abstract. 
The proofs, the results of \Cref{sec:separable}, and the examples in \Cref{sec:combinatorial}, are all new to this paper.

\section{Motivating example: The affine symmetric group} \label{sec:NiceIntro}
In this section, we illustrate our results by seeing how they apply to the affine symmetric group $\tS_n$.
Some of the notation in this section will be replaced with more general notation in the following sections.

The symmetric group $S_n$ is the group of permutations of $[n]\coloneqq \{ 1, 2, \ldots, n \}$.
Put $T = \{ (i, j) : 1 \leq i < j  \leq n \}$. For $\sigma \in S_n$, the set of \newword{inversions} of $\sigma$ is
\[ I(\sigma) : = \{ (i,j) \in T : \sigma^{-1}(i) > \sigma^{-1}(j) \} . \]

The (right) weak order on $S_n$ is the partial order on $S_n$ defined by $\sigma \leq \tau$ if and only if $I(\sigma) \subseteq I(\tau)$. 
Surprisingly, weak order is a complete lattice, meaning that every subset of $S_n$ has a unique least upper bound (its \newword{join}) and a unique greatest lower bound (its \newword{meet}).

The group $\tS_n$ is the group of bijections $\widetilde{\sigma} : \ZZ \to \ZZ$, obeying the conditions:
\[ \widetilde{\sigma}(i+n) = \widetilde{\sigma}(i) + n \quad\text{ and }\quad \sum_{k=1}^n (\widetilde{\sigma}(k)-k) = 0. \]
Put $\tT = \{ (i,j) \in \ZZ^2/{\big(} \ZZ \cdot (n,n) {\big)} : i < j,\ i \not\equiv j \bmod n \}$. To be clear, $\ZZ^2/{\big(} \ZZ \cdot (n,n) {\big)}$ is the quotient of $\ZZ^2$ by the subgroup  $\ZZ \cdot (n,n)$. We will abuse notation by referring to a coset $(i,j)+\ZZ\cdot(n,n)$ by a representative element $(i,j)$.
For $\widetilde{\sigma} \in \tS_n$, the set of \newword{inversions} of $\widetilde{\sigma}$ is
\[ I(\widetilde{\sigma}) : = \{ (i,j) \in \tT : \widetilde{\sigma}^{-1}(i) > \widetilde{\sigma}^{-1}(j) \} . \]
We once again define weak order on $\tS_n$ by $\widetilde{\sigma} \leq \widetilde{\tau}$ if and only if $I(\widetilde{\sigma}) \subseteq I(\widetilde{\tau})$. 

This time, weak order is only a complete meet-semilattice. This means that, if a subset $\cX$ of $\tS_n$ has an upper bound, then it has a least upper bound, but no upper bound need exist at all. 
(Every non-empty subset of $\tS_n$ continues to have a greatest lower bound.) 

Weak order has many applications, for example in the representation theory of quivers and in the structure of cluster algebras, and for many of these purposes we would like to extend the weak order on $\tS_n$ to some larger poset which is a complete lattice. 
It is natural to imagine doing this by defining some collection of subsets of $\tT$ which would play the role of ``generalized inversion sets''. 
Such a definition was proposed and studied by Matthew Dyer~(see, e.g., \cite{Dyer2019}); we will describe it here for the case of $S_n$ and $\tS_n$ and for a general Coxeter group in the next section.

We define a subset $J$ of $T$ to be \newword{closed} if, whenever $1 \leq i < j < k \leq n$ and $(i,j)$ and $(j,k) \in J$, then $(i,k) \in J$.

We define a subset $J$ of $\tT$ to be \newword{closed} if it obeys the four conditions
\begin{enumerate}
	\item If $i<j<k$ with $(i,j)$ and $(j,k) \in J$, and $i\not\equiv k \bmod n$, then $(i,k) \in J$.
	
	\item If $i<j<k$ with $(i,j)$ and $(j,k) \in J$, and $i\equiv k \bmod n$, then $(i,j+\ell n)$ and $(j,k+\ell n) \in J$ for all $\ell\geq 0$.
	
	\item If $i<j<k$ with $(i,j)$ and $(i,k) \in J$, and $j\equiv k \bmod n$, then $(i,j+\ell n) \in J$ for all $\ell$ such that $j \leq j+\ell n \leq k$.
	
	\item If $i<j<k$ with $(i,k)$ and $(j,k)\in J$, and $j\equiv k \bmod n$, then $(i+\ell n,k)\in J$ for all $\ell$ such that $i\leq i+\ell n\leq j$.
\end{enumerate}
Note that conditions (3) and (4) are equivalent; we have included both for clarity.
We define $J$ to be \newword{coclosed} if the complement of $J$ is closed, and we define $J$ to be \newword{biclosed} if $J$ is both closed and coclosed. 
We note that the intersection of any collection of closed sets is closed, and, for an arbitrary subset $K$ of $T$ or $\tT$, we define the \newword{closure} $\overline{K}$, to be the intersection of all closed sets containing $K$. 
Similarly, the union of any collection of coclosed sets is coclosed; we define the \newword{interior} $K^{\circ}$ to be the union of all coclosed sets contained in $K$.

The following is a special case of a result for all Coxeter groups (see, e.g., \cite{Dyer2019}):
\begin{proposition}
	The finite biclosed subsets of $T$ (respectively $\tT$) are exactly the sets of inversions of the permutations in $S_n$ (respectively $\tS_n$).
\end{proposition}

The set $T$ is finite, so all subsets of $T$ are finite. But $\tT$ is infinite, so it is natural to consider all biclosed subsets of $\tT$, without the assumption of finiteness. 

The following is a special case of a conjecture by Matthew Dyer. Dyer tells us that he already checked this case; we will give the first published proof in \Cref{sec:combinatorial}.
\begin{theorem} \label{AtildeLattice}
	The collection of all biclosed subsets of $\tT$, ordered by containment, is a complete lattice. For any collection of biclosed sets $\cX$, we have $\bigvee \cX = \overline{\bigcup_{J \in \cX} J}$ and $\bigwedge \cX = \left( \bigcap_{J \in \cX} J \right)^{\circ}$. 
\end{theorem}

We now give a combinatorial model for the biclosed subsets of $\tT$.
Define a \newword{translationally invariant total order} on $\ZZ$ to be a total order $\prec$ obeying the condition that $i \prec j$ if and only if $i+n \prec j+n$.
Each translationally invariant total order $\prec$ defines a biclosed subset $I(\prec)$ of $\tT$ by $I(\prec)  = \{ (i,j) \in \tT : i \succ j \}$. 

It turns out that any biclosed subset of $\tT$ is of this form, but not in a unique way. 
To see the issue, take $n=2$, and consider the following total orders on $\ZZ$:
\[ \cdots\prec_1 -5 \prec_1 -3 \prec_1 -1 \prec_1 1 \prec_1 3 \prec_1 5 \prec_1 \cdots \prec_1 -4 \prec_1 -2 \prec_1 0 \prec_1 2 \prec_1 4 \prec_1 \cdots \]
\[ \cdots\prec_2 -5 \prec_2 -3 \prec_2 -1 \prec_2 1 \prec_2 3 \prec_2 5 \prec_2 \cdots \prec_2 4 \prec_2 2 \prec_2 0 \prec_2 -2 \prec_2 -4 \prec_2 \cdots . \]
In other words, both $\prec_1$ and $\prec_2$ put all the odd numbers before all the even numbers, but $\prec_1$ preserves the standard ordering within each parity class whereas $\prec_2$ reverses the even numbers. 
Then $I(\prec_1) = I(\prec_2) = \{ (i,j) : i<j,\ i \equiv 0 \bmod 2,\ j \equiv 1 \bmod 2 \}$. 


\begin{theorem} \label{AtildeModel}
	Biclosed subsets of $\tT$ are in bijection with translationally invariant total orders, modulo reversing the order on intervals of the form
	\[ \cdots \prec k-3n \prec k-2n \prec k-n \prec k \prec k+n \prec k+2n \prec k+3n \prec \cdots.  \]
\end{theorem}

The main result of this paper is not simply \Cref{AtildeModel}, but a generalization of this to all affine Coxeter groups.
In the next section, we will introduce the vocabulary necessary to prove our results.

\section{Background}\label{Background}

\subsection{Coxeter groups and root systems}
Let $(m_{ij})_{i,j=1}^n$ be a symmetric matrix with $m_{ii}=1$ for all $i$, and  $m_{ij} = m_{ji} \in \{ 2,3,4,\ldots, \infty \}$ for $i\neq j$. 
Such a matrix is called a \newword{Coxeter matrix}. A \newword{Coxeter group} is a group with a presentation of the form
\[ W = {\Big\langle} s_1,\ldots,s_n {\Big|} (s_is_j)^{m_{ij}} = 1, \ 1\leq i,j\leq n {\Big\rangle} \]
for some Coxeter matrix $(m_{ij})$. In this case we say $(W,S)$ is a \newword{rank $n$ Coxeter system}, where $S=\{s_1,\ldots,s_n\}$ is the set of \newword{simple reflections}. The \newword{length} $\ell(w)$ of an element $w$ is the minimal number $k$ of simple reflections in any word $s_1\cdots s_k = w$. We further write
\[T=\{wsw^{-1} : w\in W, s\in S \}\]
for the set of \newword{reflections} in $W$.

Associated to each Coxeter group is a faithful $n$-dimensional representation $V$ called the \newword{geometric} or \newword{reflection representation}. Fix a (real-valued) \newword{Cartan matrix} $(A_{ij})_{i,j=1}^n$ such that:
\[ \begin{array}{r@{\ }c@{\ }l@{\ }l@{\qquad}l}
	A_{ii} &=& 2 & \\
	A_{ij} &=& 0 && \text{for} \ m_{ij} = 2 \\
	A_{ij} A_{ji} &=& 4 \cos^2 \tfrac{\pi}{m_{ij}},& A_{ij} < 0 &\text{for} \  2 < m_{ij} < \infty \\
	A_{ij} A_{ji} &\geq& 4 ,&  A_{ij} < 0 & \text{for} \ m_{ij} = \infty. 
\end{array} \]


The Cartan matrix $A$ is called \newword{crystallographic} if every $A_{ij}$ is an integer.
Our primary results are for crystallographic Cartan matrices, although we will state some early results in the generality of real Cartan matrices.
If $A$ is crystallographic, then we must have $4 \cos^2 \tfrac{\pi}{m_{ij}} \in \ZZ$ for $m_{ij} < \infty$, and hence $m_{ij} \in \{ 2,3,4,6,\infty \}$ for all $i \neq j$.

The \newword{geometric representation} is a representation $V$ with basis $\a_1$, \dots, $\a_n$. We will also want to work with the dual representation $V^{\ast},$ which has dual basis $\omega_1$, $\omega_2$, \dots, $\omega_n$. We write $\langle\cdot,\cdot\rangle$ for the pairing $V^*\times V\to \RR$. We put $\alpha_j^{\vee} = \sum_i A_{ij} \omega_i$, so $\langle \alpha_j^{\vee},\alpha_i \rangle = A_{ij}$.  (Note that $\alpha_i^{\vee}$ need not be a basis of $V^{\ast}$.)
Then $W$ acts on $V$ by $s_i(v) = v - \langle \alpha_i^{\vee}, v \rangle \alpha_i$ and acts on $V^{\ast}$ by the dual formula $s_i(f) = f - \langle f, \alpha_i \rangle \alpha_i^{\vee}$. 
The representations $V$ and $V^{\ast}$ are faithful representations of $W$; see~\cite[Section 4.1-4.3]{Bjorner2005}.

The (real) \newword{roots} are the vectors of the form $w \alpha_i$ for $w \in W$, $1 \leq i \leq n$. Each root is either \newword{positive}, meaning in the positive linear span of the $\alpha_i$, or \newword{negative}, meaning the negation of a positive root. We write $\Phi$ and $\Phi^+$ for the sets of roots and positive roots respectively; these are called the \newword{root system} and the \newword{positive root system}. 

The reflections in $T$ are precisely those elements of $W$ which act on $V$ and $V^{\ast}$ by involutions fixing a codimension $1$ subspace. 
Specifically, for each $t \in T$, there is a positive root $\beta_t$ and positive co-root $\beta^{\vee}_t$ such that $t(v) = v - \langle \beta_t^{\vee}, v \rangle \beta_t$. We will assume $\beta_t$ and $\beta^\vee_t$ are uniquely determined by $t$ (i.e., that our root system is \newword{reduced}); 
this is true for all root systems considered in this paper. Conversely, for any $\beta \in \Phi$ there is a unique reflection $t_\beta \in T$ such that $t_\beta \beta=-\beta$, called the \newword{reflection over $\beta$}.

We are interested in the weak order on $W$. To this end, define the set of (left) \newword{inversions} of a group element $w\in W$ to be
\[N(w) = \{\beta_t : t\in T,~ \ell(tw) < \ell(w) \}. \]
The inversion set also admits a geometric description: Take a point $f$ in $V^{\ast}$ such that $\langle f, \alpha_i \rangle >0$ for each $i$, then $\beta$ is an inversion of $w$ if and only if $\langle w f, \beta \rangle < 0$. We say $w\leq v$ in \newword{weak order} if $N(w)\subseteq N(v)$. 


%
%
%
%
%
%
%
A subset $\Phi'$ of $\Phi$ is called a \newword{subsystem} of $\Phi$ if $\Phi'$ is preserved by reflections over any $\beta\in \Phi'$. We define $\Phi'_+ = \Phi'\cap \Phi^+$ and $\Phi'_-=\Phi'\cap \Phi^-$. The following is a collection of results from \cite{Dyer1991} and \cite{Dyer1990}.

\begin{proposition}
	Let $\Phi'$ be a subsystem of $\Phi$. There is a unique minimal set $\Pi'\subseteq \Phi'_+$ such that 
	\[ \RR_{\geq 0}\Pi'\cap \Phi' = \Phi'_+ = -\Phi'_-. \]
	Let $W'$ be the subgroup of $W$ generated by reflections over the roots in $\Phi'$, and let $S'$ be the reflections over elements of $\Pi'$. Then $(W',S')$ is a Coxeter system, and the bijection between $T$ and $\Phi^+$ restricts to a bijection between $\Phi'_+$ and reflections in $W'$. Furthermore,
	\[ S' = \{t\in T : N(t)\cap \Phi' = \{\beta_t\} \}. \]
	If additionally the span of $\Phi'$ is 2-dimensional, then $|\Pi'|=2$. 
\end{proposition}

The set $\Pi'$ is called a \newword{base} for $\Phi'$ and its elements are called the \newword{simple roots} of $\Phi'$. The cardinality of $\Pi'$ is called the \newword{rank} of $\Phi'$.
The group $W'$ is called the \newword{reflection subgroup} generated by $\Phi'$. 
We say a subsystem $\Phi'$ is \newword{full} if whenever $\alpha$ and $\beta\in \Phi'$, it follows that any $\gamma$ in the span $\RR\alpha+\RR\beta$ is also in $\Phi'$. The \newword{type} of $\Phi'$ is a description of its Cartan matrix. If $X_n$ is a Cartan type, then a full subsystem of $\Phi$ with type $X_n$ is called an \newword{$X_n$-subsystem}. (E.g., an $A_2$-subsystem consists of roots $\alpha,\beta,\alpha+\beta$ and their negations, such that no other elements of $\Phi$ are in their span.) 

\begin{remark}
	In general, if $\Phi'$ is a subsystem of $\Phi$, then the elements of $\Pi'$ may be linearly dependent. In other words, the rank of $\Phi'$ might be larger than $\dim \RR\Phi'$. If $\Phi'$ is finite or the span of $\Phi'$ is 2-dimensional, then this doesn't happen. 
	In general, there even exist infinite rank subsystems of finite rank root systems; however, all subsystems of the finite rank root systems considered here are finite rank \cite{Dyer2011b}.
\end{remark}

A root system is \newword{indecomposable} when it cannot be written as the disjoint union of two nonempty full subsystems. A Coxeter group with an indecomposable root system is called \newword{irreducible}. Every root system can be written uniquely as a disjoint union of indecomposable subsystems, and each Coxeter group can be written uniquely as a product of irreducible factors.


\subsection{Finite and affine Coxeter groups} \label{FiniteAffineGroups}
Let $(A_{ij})$ be a crystallographic Cartan matrix with Coxeter group $W_0$ such that the associated root system $\Phi_0$ is finite and indecomposable. 
Such matrices are classified by \newword{Dynkin diagrams} (see \Cref{fig:finite}); we will say that $\Phi_0$ or $W_0$ is of \newword{type $X_n$} where $n$ is the rank of $\Phi_0$ and the symbol $X_n$ is one of $\{ A_n, B_n, C_n, D_n, E_6, E_7, E_8, F_4, G_2 \}$.
There is a unique root $\theta\in \Phi_0^+$ such that $\theta-\alpha$ is in the positive span of $\Phi_0^+$ for any $\alpha\in \Phi_0^+$. The root $\theta$ is called the \newword{highest root} of $\Phi_0$.


\begin{figure}
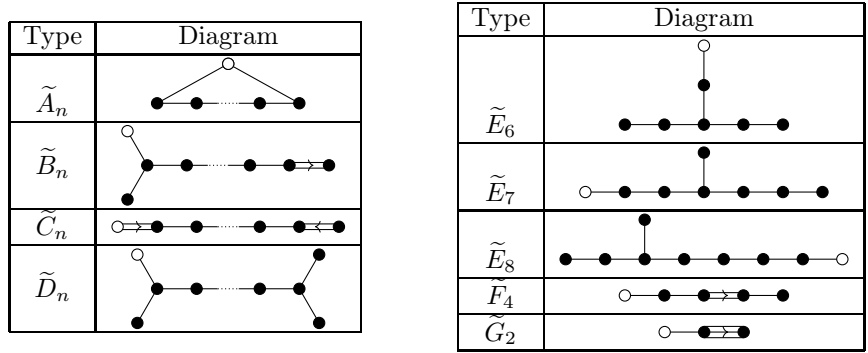

	\centering
	\begin{tabular}{|c|c|}
		\hline
		Type & Diagram \\
		\hline
		$\tA_n$ & \dynkin[scale=1.5,extended]{A}{} \\ 
		\hline
		$\tB_n$ & \dynkin[scale=1.5,extended]{B}{} \\
		\hline
		$\tC_n$ & \dynkin[scale=1.5,extended]{C}{} \\
		\hline
		$\tD_n$ & \dynkin[scale=1.5,extended]{D}{} \\
		\hline
	\end{tabular}
	\hspace{3em}
	\begin{tabular}{|c|c|}
		\hline
		Type & Diagram \\
		\hline
		$\widetilde E_6$ & \dynkin[scale=1.5,extended]
		{E}{6} \\
		\hline
		$\widetilde E_7$ & \dynkin[scale=1.5,extended]{E}{7} \\
		\hline
		$\widetilde E_8$ & \dynkin[scale=1.5,extended]{E}{8} \\
		\hline
		$\widetilde F_4$ & \dynkin[scale=1.5,extended]{F}{4} \\
		\hline
		$\widetilde G_2$ & \dynkin[scale=1.5,gonality=6,extended]{G}{2} \\
		\hline
	\end{tabular}
	\caption{The Dynkin diagrams associated to affine indecomposable root systems. The extra, unfilled, node corresponds to the extra simple root $\alpha_0$. The remaining filled nodes form the Dynkin diagram for the associated finite crystallographic root system.}\label{fig:finite}
\end{figure}

From this data, we construct a new Coxeter group $W$, called the \newword{affine Coxeter group} of type $\tX_n$. 
The affine symmetric group $\tS_n$ from the introduction is the affine Coxeter group of type $\tA_{n-1}$. 
See~\cite{Bjorner2005} or \cite{Kac1983} for more.

The affine Coxeter group $W$ has simple generators $s_0$, $s_1$, \dots, $s_n$ where, for $i$, $j \geq 1$, we have the same Coxeter matrix $m_{ij}$ as in $W_0$ and, for $j \geq 1$, we take $m_{0j} = m_{j0}$ to be the order of $t_{\theta} s_j$ in $W_0$.
Thus, there is a map $W \to W_0$ sending $s_i$ to $s_i$ and $s_0 \to t_{\theta}$, as well as an inclusion $W_0 \into W$; making $W$ into a semidirect product $W_0 \ltimes P$. 
It turns out that $P$ is a free abelian group of rank $n$.

The corresponding root system can be described as follows:
Let $V_0$ be the ambient $n$-dimen\-sional vector space of $\Phi_0$ and let $V \coloneqq  V_0 \oplus \RR\delta$ be an $(n+1)$-dimensional vector space, with new basis vector $\delta$. 
Then $\Phi$ is the set
\[ \Phi = \{ \alpha+k\delta : \alpha\in\Phi_0,\ k\in\ZZ \}. \]
The simple roots of $\Phi$ are $\alpha_1$, \ldots, $\alpha_n$, $\alpha_0$, where $\alpha_1$, \dots, $\alpha_n$ are the simple roots of $\Phi_0$ and $\alpha_0=\delta-\theta$. The affine Dynkin diagrams of the possible resulting root systems are shown in \Cref{fig:finite}.
We write $\pi$ for the map $\Phi \to \Phi_0$ sending $\beta+k \delta$ to $\beta$, for $\beta \in \Phi_0$.

For each pair $\pm \beta$ of roots in $\Phi_0$, there is a type $\tA_1$ root subsystem of $\Phi$, whose elements are $\{ \pm \beta + k \delta : k \in \ZZ \}$. 
Each root of $\Phi$ lies in exactly one such $\tA_1$ subsystem. 

We now recall the definition of the \newword{Coxeter fan} of a finite Coxeter group $W_0$. This is a complete fan living in the dual representation $V^{\ast}_0$ of $W_0$. The set $F_0$ of $f\in V^{\ast}_0$ such that $\langle f,\alpha\rangle\geq 0$ for all $\alpha\in \Phi_0^+$ is called the \newword{dominant chamber}. A \newword{chamber} is any cone of the form $wF_0$ for $w\in W_0$. The Coxeter fan consists of all cones formed by intersections of chambers. These cones are called the \newword{faces} of the fan. Let $\faces(W_0)$ denote the set of faces of the Coxeter fan of $W_0$. The action of $W_0$ on $V^{\ast}_0$ permutes the elements of $\faces(W_0)$. 

When we defined $W$, we stated that there was a map $W \to W_0$, so $W$ acts on $\faces(W_0)$ through this map. 
We now describe the action of $W$ on $\faces(W_0)$ more geometrically: we embed $V_0^\ast$ as a codimension 1 subspace of $V^\ast$ by identifying the co-roots $\alpha_1^\vee$, \dots, $\alpha_n^\vee$ in $V_0^\ast$ with the corresponding co-roots in $V^\ast$. This identification is $W$-equivariant and identifies $V_0^\ast$ with the subspace $\delta^\perp \subset V^\ast$. So we think of the $W_0$--Coxeter fan as living in a hyperplane in $V^*$, and this perspective is compatible with the $W$-action.

Given $F\in \faces(W_0)$, we write $\Phi_0^F$ for set of roots in $\Phi_0$ which are orthogonal to $F$, and we write $\Phi_F$ for the set of roots $\beta \in \Phi$ with $\pi(\beta) \in \Phi_0^F$. 
Let $W_F$ be the reflection group generated by the reflections over the roots $\beta$ in $\Phi_F$; we call $W_F$ a \newword{face subgroup} of $W$.
If $F$ and $F'$ are faces of the $W_0$--Coxeter fan with the same linear span, then $W_F=W_{F'}$. Hence, face subgroups are indexed by spans of faces of the $W_0$--Coxeter fan or, equivalently, parabolic subgroups of $W_0$.

\subsection{Closed and biclosed sets} \label{ClosedBiclosedSection}
A subset $J$ of $\Phi^+$ is called \newword{closed} if, for any three roots $\alpha$, $\beta$ and $\gamma$ with $\gamma \in \RR_{>0} \alpha + \RR_{>0} \beta$, if $\alpha$ and $\beta \in J$ then $\gamma \in J$.
A subset $J$ of $\Phi^+$ is called \newword{coclosed} if $\Phi^+ \setminus J$ is closed. 
The set $J$ is \newword{biclosed} if it is both closed and coclosed.
As in the introduction, for a subset $A$ of $\Phi^+$, we define the \newword{closure} $\overline{A}$ to be the intersection of all closed sets containing $A$ and we define the \newword{interior} $A^{\circ}$ to be the union of all coclosed sets contained in $A$. 

If $J$ is a subset of $\Phi^+$ and $\Phi'$ is a subsystem of $\Phi$, then we say $J$ is \newword{finite} in $\Phi'$ if $|J\cap \Phi'|$ is finite. We say $J$ is \newword{cofinite} in $\Phi'$ if $|\Phi'_+\setminus J|$ is finite.

To give some intuition for what biclosed sets look like, we will introduce a notation for describing rank 2 systems inside $\Phi^+$.
Each rank two subsystem can be linearly ordered such that,  if $\gamma \in \RR_{>0} \alpha + \RR_{>0}\beta$, then $\gamma$ appears in between $\alpha$ and $\beta$ in the order. 
The linear order is characterized by this property up to reversal.
If $\gamma \in \RR_{>0} \alpha + \RR_{>0}\beta$, then we say that $\gamma$ is \newword{between} $\alpha$ and $\beta$, and we abbreviate this to $\alpha \to \gamma \to \beta$; note that the order is only defined up to reversal, so we could equally well write $\beta \to \gamma \to \alpha$.

Closed sets in $\Phi^+$ may now be characterized as those sets $B$ such that for each full rank 2 subsystem $\Phi'$ and for any ordering on $\Phi'_+$ as above, the intersection $B\cap \Phi'$ is an interval in the order. Similarly, biclosed sets in $\Phi^+$ are exactly those sets such that $B\cap \Phi'$ is either a down-set in the order, or an up-set in the order, for all full rank 2 subsystems $\Phi'$. 

\begin{remark}
	There are global total orders on $\Phi^+$ which restrict to an order as above on each rank two subsystem.
	Such a total order is called a \newword{reflection order}; we do not make use of the theory of reflection orders. The interested reader can refer to \cite{Bjorner2005} or \cite{Dyer2019}.
\end{remark}

\begin{eg}
	Consider \Cref{fig:reflorder}, a diagram of the $\tA_1$ root system. The set $\{\alpha_0+n\delta : n\geq 0\}$, shaded blue in the figure, is a biclosed set since it is a down-set in the ordering shown below the figure. Other biclosed sets include $\{\alpha_1, \alpha_1+\delta\}$ and $\Phi^+\setminus \{\alpha_0\}$, which are both up-sets in the order. The former of these is finite (since it has 2 elements) and the latter is cofinite (since its complement has 1 element). The biclosed set shown in the figure is neither finite nor cofinite.
\end{eg}


\begin{figure}
	\centering
	\begin{tikzpicture}
		\begin{scope}[shift={(0,-2)}]
			\clip (-2.5,0) rectangle (2.5,4);
			\foreach \x in {0,1,...,20} {
				\draw[-stealth] (0,0) -- ({atan2(\x,\x+1)+45}:{sqrt(\x^2+(\x+1)^2)});
				\draw[blue,-stealth] (0,0) -- ({atan2(\x+1,\x)+45}:{sqrt(\x^2+(\x+1)^2)});
			};
			\fill[blue] (0,0) -- ({atan2(20+1,20)+45}:{sqrt(20^2+(20+1)^2)}) -- (90:4);
			\fill[black] (0,0) -- ({atan2(20,20+1)+45}:{sqrt(20^2+(20+1)^2)}) -- (90:4);
			\draw ({atan2(0,1)+45}:{sqrt((1-1)^2+(1)^2)}) node[right] at ++(0,0) {$\alpha_1$};
			\draw ({atan2(1,1-1)+45}:{sqrt((1-1)^2+(1)^2)}) node[left] at ++(0,0) {$\alpha_0$};
			\draw ({atan2(2-1,2)+45}:{sqrt((2-1)^2+(2)^2)}) node[right] at ++(0,0) {$\alpha_{1}+\delta$};
			\draw ({atan2(2,2-1)+45}:{sqrt((2-1)^2+(2)^2)}) node[left] at ++(0,0) {$\alpha_{0}+\delta$};
			\draw ({atan2(3-1,3)+45}:{sqrt((3-1)^2+(3)^2)}) node[right] at ++(0,0) {$\alpha_{1}+2\delta$};
			\draw ({atan2(3,3-1)+45}:{sqrt((3-1)^2+(3)^2)}) node[left] at ++(0,0) {$\alpha_{0}+2\delta$};
		\end{scope}
		\draw node at ({sqrt(2)/2+.3-.05},2.1) {$\vdots$};
		\draw node at ({-sqrt(2)/2-.3},2.1) {$\vdots$};
	\end{tikzpicture}
	\[\alpha_0\to \alpha_0+\delta \to \alpha_0+2\delta \to \cdots \to \alpha_1+2\delta \to \alpha_1+\delta \to \alpha_1\]
	\caption{The positive roots in the $\tA_1$ root system. 
		A possible reflection order is shown below. 
		An example of a biclosed set is shaded in blue.
	} \label{fig:reflorder}
\end{figure}
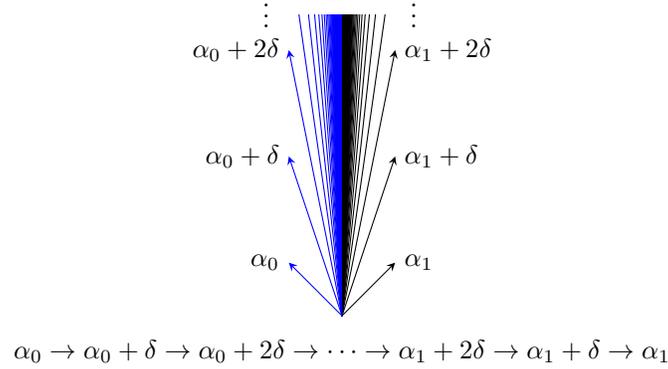

Recall that the set of inversions of a Coxeter group element is a finite set which can be described as $N(w) = \{ \beta \in \Phi^+ : \langle wf, \beta \rangle < 0 \}$, where $f \in V^{\ast}$ obeys $\langle f, \alpha_i \rangle >0$ for all $i$. 
From this geometric description, it is clear that $N(w)$ is biclosed for any $w \in W$. This characterizes the finite biclosed sets:

\begin{proposition}[\cite{Dyer2019}]\label{FiniteBiclosed}
	A set $B$ of positive roots is finite and biclosed if and only if there exists $w\in W$ such that $B = N(w)$. In this case, $w$ is uniquely defined.
\end{proposition}

There is another way to describe biclosed sets, using subsets of $\Phi$ rather than $\Phi^+$: Let $B$ be a subset of $\Phi^+$ and let $D(B) = B \cup (- (\Phi^+ \setminus B ))$. So, for each $\beta \in \Phi^+$, we put $\beta$ into $D(B)$ if $\beta \in B$ and we put $- \beta$ into $D(B)$ if $\beta \not\in B$.
\begin{lemma} \label{lem:DoubleCondition}
	The set $B$ is biclosed if and only if there do not exist $\alpha$, $\beta$, $\gamma \in D(B)$ and positive real scalars $a$, $b$, $c$ such that $a \alpha + b \beta + c \gamma = 0$.
\end{lemma}

\begin{proof}
	If $B$ is not closed, then there are $\alpha$, $\beta \in B$ and $\gamma \not \in B$ with $\gamma \in \RR_{>0} \alpha + \RR_{>0} \beta$. Then $\alpha$, $\beta$ and $- \gamma \in D(B)$ and there is a linear relation between $\alpha$, $\beta$ and $- \gamma$ as above.
	
	If $B$ is not coclosed, then there are $\alpha$, $\beta \not\in B$ and $\gamma  \in B$ with $\gamma \in \RR_{>0} \alpha + \RR_{>0} \beta$. Then $-\alpha$, $-\beta$ and $\gamma \in D(B)$ and there is a linear relation between $-\alpha$, $-\beta$ and $\gamma$ as above.
	
	Conversely, suppose there are $\alpha$, $\beta$ and $\gamma \in D(B)$ as above. Then $\alpha$, $\beta$ and $\gamma$ cannot all be positive roots or all negative roots; rename them so that either $\alpha$ and $\beta \in \Phi^+$ and $\gamma \in (- \Phi^+)$, or else $\alpha$ and $\beta \in (- \Phi^+)$ and $\gamma \in \Phi^+$. 
	In the first case, $\alpha$ and $\beta \in B$ and $(- \gamma) \not\in B$, showing that $B$ is not closed; 
	in the second case, $(-\alpha)$ and $(-\beta) \not\in B$ and $\gamma \in B$, showing that $B$ is not coclosed.
\end{proof}

We conclude this section by connecting the general  notion of biclosure  to the biclosed subsets from \Cref{sec:NiceIntro}.
Let $e_1$, $e_2$, \ldots, $e_n$ be the standard basis of $\RR^n$, then $\{ e_i-e_j : 1 \leq i < j \leq n \}$ form a root system of type $A_{n-1}$, which spans an $(n-1)$-dimensional subspace of $\RR^n$.
The rank 2 subsystems of $\Phi^+$ come in two forms: The type $A_2$ subsystems are 
\[ (e_i-e_j) \to (e_i-e_k) \to (e_j-e_k) \] 
for $i<j<k$, and the type $A_1 \times A_1$ subsystems are 
\[ (e_i - e_j) \to (e_k - e_{\ell}) \] 
for $i < j$ and $k < \ell$ distinct.
The closure conditions  for the $A_2$-subsystems are the conditions from \Cref{sec:NiceIntro}; the closure conditions for the $A_1 \times A_1$-subsystems are trivial, since they only have two positive roots. 

Similarly, define $\widehat{V}$ to be the $n+1$-dimensional vector space spanned by a collection of vectors $\te_i$, for $i \in \ZZ$, modulo the relation $\te_{i+n} - \te_i = \te_{j+n} - \te_j$ for all $i$, $j$. Then the vectors $\{ \te_i - \te_j : i<j,\ i \not \equiv j \bmod n \}$ form a copy of the positive $\widetilde{A}_{n-1}$-system, inside an $n$-dimensional subspace of $\widehat{V}$. There are three types of full rank 2 subsystems: For $i<j<k$ all distinct modulo $n$, we have an $A_2$-subsystem 
\[ (\te_i-\te_j) \to (\te_i-\te_k) \to (\te_j-\te_k). \] 
For $i < j < i+n$ we have an $\tA_1$-subsystem on the roots $\{ \te_{i}- \te_{j+pn} :  p\geq 0 \}\cup \{ \te_{j} - \te_{i+pn} : p \geq 1 \}$ with ordering
\[ (\te_i - \te_j) \to (\te_i - \te_{j+n}) \to (\te_i-\te_{j+2n}) \to \cdots \to (\te_{j}-\te_{i+3n}) \to (\te_{j}-\te_{i+2n}) \to (\te_j-\te_{i+n}) .\] 
For $i < j$, $k < \ell$ distinct modulo $n$, we have an $A_1 \times A_1$-subsystem 
\[ (\te_i - \te_j) \to (\te_k - \te_{\ell}) .\] 
The closure conditions for  the first two types of subsystems are the conditions from \Cref{sec:NiceIntro}; the closure conditions  for  the 
$A_1 \times A_1$-subsystems are trivial.


\section{A combinatorial model for biclosed sets in affine type}\label{sec:AffineModels}

We now give a description of the biclosed sets for affine Coxeter groups, generalizing \Cref{AtildeModel}. 
Many of these results were proved independently by Matthew Dyer in unpublished work~ \cite{Dyerpreprint}. 

Fix an affine root system $\Phi$ with associated finite subsystem $\Phi_0$ for the duration of the section. When we say ``biclosed set'' here, we mean a biclosed set in $\Phi^+$.

\subsection{The commensurability relation}
\begin{lemma}\label{lem:sufficientlylarge}
	Let $B$ be a biclosed subset in $\Phi^+$ and let $\beta$ be a root in $\Phi_0$. Then either $\beta + k \delta \in B$ for all sufficiently large $k$, or else  $\beta + k \delta \not\in B$ for all sufficiently large $k$.
\end{lemma}

\begin{proof}
	The intersection of $\Phi^+$ with $\RR \beta + \RR \delta$ is 
	\[ \beta \to \beta+\delta \to \beta+2 \delta \to \cdots \to 2 \delta - \beta \to \delta - \beta \]
	if $\beta$ is a positive root, and is
	\[ \beta+\delta \to \beta+2 \delta \to \cdots \to 2 \delta - \beta \to \delta - \beta \to (-\beta) \]
	if $\beta$ is a negative root. 
	Either way, the roots $\beta+k \delta$ occur in linear order as $k$ grows, so the result follows.
\end{proof}

This lemma allows us to make the following definition:
For a biclosed subset $B$ of $\Phi^+$, we define 
\[ B_{\infty} = \{ \beta \in \Phi_0 : \beta+k \delta \in B \ \text{for}\ k \gg 0 \}. \]
We will classify biclosed subsets of $\Phi$ by first classifying the possible sets $B_{\infty}$ and then classifying the biclosed sets $B$ corresponding to each $B_{\infty}$.


For two sets $B$ and $C$, we define $B \oplus C$ to be the symmetric difference $(B \setminus C) \sqcup (C \setminus B)$.
We say that $B$ and $C$ are \newword{commensurable} if $B \oplus C$ is finite.
Being commensurable is clearly an equivalence relation, and we will refer to an equivalence class of biclosed sets for this equivalence relation as a \newword{commensurability class}.

\begin{lemma} \label{lem:CommensurableInfinity}
	Let $B$ and $C$ be two biclosed sets in $\Phi^+$. Then $B$ and $C$ are commensurable if and only if $B_\infty=C_\infty$.
\end{lemma}

\begin{proof}
	The positive root system $\Phi^+$ breaks up into finitely many disjoint pieces of the form $\beta + \ZZ_{\geq 0} \delta$, for $\beta \in \Phi^+_0$, and $\beta + \ZZ_{> 0} \delta$, for $-\beta \in \Phi^+_0$. So $B$ and $C$ are commensurable if and only if their intersections with $\beta + \ZZ_{\geq 0} \delta$ are commensurable.
	For each $\beta$, the intersection of $B$ with  $\beta + \ZZ_{\geq 0} \delta$ is either finite or cofinite in $\beta + \ZZ_{\geq 0} \delta$; the former if $\beta \not\in B_{\infty}$ and the latter if $\beta \in B_{\infty}$. 
	So $B_{\infty} = C_{\infty}$ if and only if $B \cap \left( \beta + \ZZ_{\geq 0} \delta \right)$ and $C \cap \left( \beta + \ZZ_{\geq 0} \delta \right)$ are either both finite or both cofinite for each $\beta \in \Phi_0$, which is the same as saying that $B$ and $C$ are commensurable.
\end{proof}

Let $C$ be a subset of $\Phi_0$. We say that $C$ is \newword{biclosed in $\Phi_0$} if, for any (positive or negative) roots $\alpha,\beta,\gamma \in \Phi_0$ such that $\gamma \in \RR_{>0} \alpha+ \RR_{>0} \beta$, then $\alpha$, $\beta\in C$ implies $\gamma\in C$, and $\alpha$, $\beta\not\in C$ implies $\gamma\not\in C$.

\begin{lemma} \label{lem:BInfinityClosed}
	Let $B$ be a biclosed subset in $\Phi^+$. Then $B_{\infty}$ is biclosed in $\Phi_0$.
\end{lemma}
\begin{proof}
	Let $\alpha$, $\beta$ and $\gamma \in \Phi_0$ such that $\gamma = a \alpha + b \beta$ for $a$ and $b \in \QQ_{>0}$. 
	
	Suppose that $\alpha$ and $\beta \in B_{\infty}$. For all sufficiently large integers $p$ and $q$, we have $\alpha + p \delta$ and $\beta + q \delta \in B$. 
	Assuming we choose $p$ and $q$ such that $ap+bq$ is an integer, $a (\alpha + p \delta) + b (\beta + q \delta) = \gamma + (ap+bq) \delta$ is a positive root and lies in a $\tA_1$ subsystem in between  $\alpha + p \delta$ and $\beta + q \delta$.
	So, for all sufficiently large $p$ and $q$ such that $ap+bq$ is an integer, $\gamma + (ap+bq) \delta$ is in $B$, and this forces $\gamma$ into $B_{\infty}$ by \Cref{lem:sufficientlylarge}.
	
	The case where $\alpha$ and $\beta \not\in B_{\infty}$ is exactly analogous.
\end{proof}
We therefore want to classify biclosed sets in $\Phi_0$. 
Once we have done this, we will return then classify which biclosed sets in $\Phi^+$ can give rise to a particular biclosed set in $\Phi_0$.

\subsection{Classification of biclosed subsets of \texorpdfstring{$\Phi_0$}{Phi\_0}}
To understand the sets which can appear as $B_\infty$ for some biclosed set $B$, 
we will use the Coxeter fan described in \Cref{FiniteAffineGroups}.
Let $F$ be a face of the $W_0$-Coxeter fan and let $f$ lie in the relative interior of $F$.
Then $\{ \beta \in \Phi_0 : \langle f, \beta \rangle < 0 \}$ and $\{ \beta \in \Phi_0 : \langle f, \beta \rangle \leq 0 \}$ are both biclosed sets. 
We now describe a more general construction that interpolates between $\{ \beta \in \Phi_0 : \langle f, \beta \rangle < 0 \}$ and $\{ \beta \in \Phi_0 : \langle f, \beta \rangle \leq 0 \}$.

Let $\Phi_0^F =  \{ \beta \in \Phi_0 : \langle f, \beta \rangle  = 0\}$. 
If $\Phi'_0$ is any union of indecomposable components of $\Phi^F_0$, we define:
\[ B_0(F, \Phi'_0) \coloneqq   \{ \beta \in \Phi_0 : \langle f, \beta \rangle  < 0\}  \cup \Phi'_0 \]
The following is the main result of {Đ}okovi\'{c}, Check, and H\'{e}e~\cite[Theorem 4]{Dokovic1994}:
\begin{theorem} \label{thm:biclosedFinite}
	The biclosed subsets of $\Phi_0$ are precisely the sets $B_0(F, \Phi'_0)$.
\end{theorem}
We remark on some differences between the statement in their paper and the one shown here. The statement of the theorem in \cite{Dokovic1994} gives only $W_0$-orbit representatives of biclosed sets. The representatives they list are exactly the sets $B_0(F,\Phi'_0)$ where $F$ is a face in the dominant chamber of the Coxeter fan. (More specifically, their $\Delta$ is a base for our $\Phi_0^F$, and their $\Delta'$ is a base for our $\Phi_0'$.) Because every element of $\faces(W_0)$ is in the $W_0$-orbit of a face in the dominant chamber, the biclosed sets they construct are the same as ours. 
A more serious difference is that their article defines closed sets to be those subsets $B$ of $\Phi_0$ such that if $\alpha,\beta\in B$, then $\alpha+\beta\in B$. This differs from our notion of closed. However, by the following Lemma, their definition gives the same biclosed sets of $\Phi_0$ as our definition.
\begin{lemma}
	Let $B$ be a subset of a finite crystallographic root system $\Phi_0$. Then $B$ is biclosed if and only if, for all $\alpha,\beta,\gamma\in \Phi_0$ such that $\alpha+\beta=\gamma$, if $\alpha$ and $\beta\in B$ then $\gamma\in B$, and if $\alpha$ and $\beta\not\in B$ then $\gamma\not\in B$.
\end{lemma}
\begin{proof}
	The forwards direction is clear. To prove the converse, it is enough to prove it for the four different types of finite rank 2 system: $A_1\times A_1$, $A_2$, $B_2$, and $G_2$. For each of these it is a straightforward check.
\end{proof}

Let $\Phi_F =  \{ \beta \in \Phi^+ : \langle F, \beta \rangle  = 0\}$; this is the root system of the face subgroup $W_F$.
The indecomposable subsytems of $\Phi_F$ are in bijection with the indecomposable subsystems of $\Phi^F_0$. We let $\Phi'$ be a union of indecomposable subsystems of $\Phi_F$. 
We now associate a ``standard'' biclosed subset of $\Phi^+$ to each $(F, \Phi')$. Let $f$ be a functional in the relative interior of $F$.
We define 
\[ B(F, \Phi') \coloneqq   \{ \beta \in \Phi^+ : \langle f, \beta \rangle  < 0\}  \cup (\Phi' \cap \Phi^+) .\]

The following is immediate from the definition:
\begin{lemma} \label{lem:BasicRepresentative}
	With the above notation, we have
	\[ B(F, \Phi')_{\infty} = B_0(F, \Phi'_0) . \]
\end{lemma}

We also verify:

\begin{lemma} \label{lem:BasicRepresentativeBiclosed}
	The sets $B(F, \Phi')$ are biclosed in $\Phi^+$.
\end{lemma}

\begin{proof}
	We abbreviate $B_-: = \{ \beta \in \Phi^+ : \langle f, \beta \rangle  < 0\}$, $B_0 \coloneqq  \{ \beta \in \Phi^+ : \langle f, \beta \rangle  = 0\}$ and $B_+: = \{ \beta \in \Phi^+ : \langle f, \beta \rangle  > 0\}$. 
	So $B(F, \Phi')$ is $B_- \cup (\Phi' \cap \Phi^+)$ and $\Phi^+ \setminus B(F, \Phi')$ is $B_+ \cup (B_0 \setminus \Phi')$. 
	Now, suppose that we have distinct roots $\alpha$, $\beta$ and $\gamma$ in $\Phi^+$ with $\gamma \in \RR_{>0} \alpha + \RR_{>0} \beta$. 
	
	If $\alpha$ and $\beta$ are in $(B_+, B_-)$ or $(B_-, B_+)$ (respectively), then the biclosure condition is automatic.
	If $\alpha$ and $\beta$ are in $(B_{\pm}, B_{\pm})$, $(B_{\pm}, B_0)$ or $(B_0, B_{\pm})$ (respectively, with all signs the same), then the condition on $\langle f, \ \rangle$ implies that $\gamma$ is in $B_{\pm}$ as well, and the biclosure condition follows. 
	The only remaining case is where $\alpha$ and $\beta \in B_0$. In this case, the condition on $\langle f, \ \rangle$ implies that $\gamma$ is in $B_0$. 
	Since the different indecomposable components of $\Phi_F$ are mutually orthogonal root systems, the condition that $\gamma \in \RR_{>0} \alpha + \RR_{>0} \beta$ implies that $\alpha$, $\beta$ and $\gamma$ are all in the same component of $\Phi_F$, so either all of them are in $\Phi'$ or none are and, in either case, the biclosed condition is satisfied.
\end{proof}

Thus, combining 
\Cref{lem:CommensurableInfinity,lem:BInfinityClosed,lem:BasicRepresentative,lem:BasicRepresentativeBiclosed} and \Cref{thm:biclosedFinite}, we see that we can classify the biclosed subsets of $\Phi$ by finding the biclosed sets which are commensurable to $B(F, \Phi')$ for each $(F, \Phi')$. 
We carry out this task in the next section.

\subsection{The commensurability class of \texorpdfstring{$B(F, \Phi')$}{B(F,Phi')}} \label{subsec:CommensurabilityClasses}
We note that $B(\{ 0 \}, \emptyset) = \emptyset$ and that the commensurability class of $B(\{ 0 \}, \emptyset)$ consists of the finite biclosed sets which, by \Cref{FiniteBiclosed}, are in bijection with the elements of $W$.
Our goal will be to generalize this by giving a bijection between the commensurability class of $B(F, \Phi')$ and the face subgroup $W_F$.

To do this, we define an action of $W$ on the class of biclosed sets, introduced in~\cite[Section 4.1]{Dyer2019}.
The easiest way to describe this action is in terms of the doubling construction from Lemma~\ref{lem:DoubleCondition}: if $w$ is in $W$ and $B$ is a biclosed set, then $w\cdot B$ is defined to be the unique biclosed set such that
\[ D(w \cdot B) = w D(B). \]
We can rewrite this formula (less naturally) so to not mention $D$: We have
\[ \begin{array}{lcl}
	w \cdot B &=& \{ w \gamma : w \gamma \in \Phi^+,\ \gamma \in \Phi^+, \ \gamma \in B \} \cup    \{ w \gamma : w \gamma \in \Phi^+,\ (-\gamma) \in \Phi^+, \ (-\gamma) \not\in B \} \\
	&=& \{ |w \gamma| : \gamma \in B \} \oplus N(w), \\
\end{array}\]
where $|\phi|$ is $\phi$ if $\phi \in \Phi^+$ and $-\phi$ if $- \phi \in \Phi^+$. With this definition, it holds that $v\cdot N(w) = N(vw)$ for any elements $v,w$ of $W$.

\begin{lemma} \label{lem:ActionFormula}
	For $w\in W_F$, we have
	\[ w\cdot B(F,\Phi') =
	B(F,\Phi')\oplus (N(w)\cap \Phi_F). \]
\end{lemma}
\begin{proof}
	In terms of the doubling construction, we need to show that 
	\[ w D(B(F, \Phi')) = D(B(F, \Phi')) \oplus \{ \pm \beta : \beta \in N(w) \cap \Phi_F \} . \]
	Given a root $\beta$, we must show it is contained in the left hand side if and only if it is contained in the right.
	
	\textbf{Case 1:} $\beta \not \in \Phi_F$. Then $\langle F,\beta \rangle \neq 0$. Since $w \in W_F$, we have $wF = F$ and $\langle F,w^{-1}(\beta) \rangle = \langle w F,\beta  \rangle = \langle F,\beta \rangle$. So $\langle F,w^{-1}(\beta) \rangle$ has the same (nonzero) sign as $\langle F,\beta \rangle$. So $w^{-1}(\beta) \in D(B(F, \Phi'))$ if and only if $\beta \in D(B(F, \Phi'))$. We see that $\beta$ is contained in the left hand side if and only if it is contained in the right.
	
	\textbf{Case 2:} $\beta \in \Phi_F$. In this case, we need to break into two cases. 
	
	\textbf{Case 2a:} $|\beta| \not\in N(w) \cap \Phi_F$. In this case, $w^{-1}(\beta)$ and $\beta$ have the same sign, and both of them are in the same indecomposable component of $\Phi_F$. So either $\beta$ and $w^{-1}(\beta)$ are both in $B(F, \Phi')$ or neither are, and we see that $\beta$ is in the left hand side if and only if it is in the right hand side.
	
	\textbf{Case 2b:} $|\beta| \in N(w) \cap \Phi_F$. In this case, $w^{-1}(\beta)$ and $\beta$ have opposite signs, and both of them are in the same indecomposable component of $\Phi_F$. 
	So exactly one of $\beta$ and $w^{-1}(\beta)$ is in $B(F, \Phi')$, and we see that $\beta$ is in the left hand side if and only if it is in the right hand side.
\end{proof}

\begin{theorem}
	The set of biclosed sets commensurable to $B(F, \Phi')$ is the $W_F$-orbit of $B(F, \Phi')$, and $W_F$ acts freely on this orbit. 
\end{theorem}

\begin{proof}
	This theorem contains three statements, which we check in turn.
	
	First, we show that, for $w \in W_F$, the biclosed sets  $w \cdot B(F, \Phi')$ and $B(F, \Phi')$ are commensurable. 
	This follows from the formula in \Cref{lem:ActionFormula}, which writes $w \cdot B(F, \Phi')$ as the symmetric difference of $B(F, \Phi')$ and a finite set.
	
	Next, we show that $W_F$ acts freely on $B(F, \Phi')$. Suppose that $w \in W_F$ with $w \cdot B(F,\Phi') = B(F, \Phi')$.
	Then \Cref{lem:ActionFormula} shows that $N(w) \cap \Phi_F = \emptyset$. 
	But $N(w) \cap \Phi_F$ is in bijection with the set of inversions of $w$ when considered as an element of the Coxeter group $W_F$, so $N(w) \cap \Phi_F = \emptyset$ only if $w = e$.
	
	Finally, let $B$ be a biclosed set commensurable with $B(F, \Phi')$. We want to show that $B = w \cdot B(F, \Phi')$ for some $w \in W_F$.
	
	Write $B = B(F, \Phi') \oplus X$. 
	We first claim that $X \subseteq \Phi_F$. 
	In other words, for $\beta \not \in \Phi_F$, we want to show that $B \cap (\RR \beta + \RR \delta) = B(F, \Phi') \cap (\RR \beta + \RR \delta)$. 
	The assumption that $\beta \not \in \Phi_F$ means that $B(F, \Phi') \cap (\RR \beta + \RR \delta)$ is either the set of positive roots of the form $\beta + k \delta$, or the set of positive roots of the form $- \beta + k \delta$. 
	We know that $B \cap (\RR \beta + \RR \delta)$ must be commensurable with this and must be biclosed in the $\tA_1$ subsystem $\pm \beta + \ZZ \delta$; the only way for this to happen is for $B \cap (\RR \beta + \RR \delta)$ to be $B(F, \Phi') \cap (\RR \beta + \RR \delta)$. 
	So we have shown that $X \subseteq \Phi_F$. 
	
	Let $\Phi_i$ be any indecomposable component of $\Phi_F$, let $W_i$ be the associated affine Coxeter group and let $X_i = X \cap \Phi_i$. 
	So $X_i$ is finite. If $\Phi_i$ is not one of the components of $\Phi'$, then $B \cap \Phi_i = X_i$; if $\Phi_i$ is one of the components of $\Phi'$, then $B \cap \Phi_i = \Phi_i \setminus X_i$. 
	Either way, we deduce that $X_i$ is a finite set which is  biclosed in $\Phi_i$.
	Therefore, $X_i = N_{W_i}(w_i)$ for some $w_i \in W_i$, where $N_{W_i}$ means to compute the inversion set within the Coxeter group $W_i$.
	Using the isomorphism $W_F \cong \prod_i W_i$, we see that $X = N_{W_F} \left( \prod_i w_i \right)$.
	Then, putting $w = \prod_i w_i$,  \Cref{lem:ActionFormula} shows that $B = w \cdot B(F, \Phi')$.
\end{proof}

Combining our results, we have the following.

\begin{theorem}\label{Parametrization}
	Let $W$ be an affine Coxeter group and $W_0$ the associated finite Coxeter group. Biclosed sets for $W$ are in bijection with triples $(F, \Phi',w)$, where $F$ is a face of the $W_0$-Coxeter fan, and $w\in W$ stabilizes $F$, and $\Phi'$ is a union of indecomposable subsystems of $\Phi_F$. Explicitly, if $f$ is a functional in the relative interior of $F$, then the set associated to the triple is
	\[ B(F, \Phi', w) = {\Big(} \left( \{\alpha+r\delta \mid r\in\ZZ,~\alpha\in\Phi_0,~ \langle f,\alpha\rangle<0\}\cap \Phi^+ \right) \sqcup \Phi'_+ {\Big)} \oplus {\Big(} N(w)\cap \Phi_F {\Big)}. \]
	The triples $(F_1, \Phi'_1, w_1)$, $(F_2,\Phi'_2,w_2)$ are in the same path component if and only if $F_1=F_2$ and $\Phi'_1=\Phi'_2$.
\end{theorem}

This is Theorem~\ref{thm:Main}, with some additional details added.

\begin{remark}
	For a general affine Coxeter group, the limit weak order of Lam and Pylyav\-skyy~\cite{LP} corresponds to triples $(F, \emptyset, w)$. This was shown by Weijia Wang in \cite{Wang2018}.
\end{remark}

We can use our parametrization to describe the poset of biclosed sets in a commensurability class in terms of weak order on the face subgroup.  

\begin{corollary}
	Let $F$ and $\Phi'$ be as above. 
	Let the components of $\Phi_F$ be $\Phi_1 \sqcup \Phi_2 \sqcup \cdots \sqcup \Phi_r$, with corresponding Coxeter groups $W_1$, $W_2$, \dots, $W_r$. 
	Let $\Phi' = \bigsqcup_{i \in I} \Phi_i$ for some index set $I \subseteq [r]$.
	The partially ordered set of biclosed sets which are commensurable with $B(F, \Phi')$ is isomorphic to $\prod_{i \in I} W_i^{\op} \times \prod_{i \in [r] \setminus I} W_i$.
	Here $W_i$ is the weak order on the Coxeter group $W_i$ and $W_i^{\op}$ is the opposite of this weak order.
\end{corollary} 

\begin{proof}
	From \Cref{lem:ActionFormula}, the biclosed sets commensurable with $B(F, \Phi')$ are of the form 
	\[B(F, \Phi') \oplus X,\] 
	where $X \subset \Phi_F$ corresponds to the inversions of $w \in W_F$.
	Writing $w = (w_1, w_2, \ldots, w_r)$ for $w_i \in W_i$, this is $B(F, \Phi') \setminus \bigcup_{i \in I} N_{W_i}(w_i) \cup \bigcup_{i \in [r] \setminus I} N_{W_i}(w_i)$, where $N_{W_i}(w_i)$ is the set of inversions of $w_i$ as an element of $W_i$.
	
	Containment of inversion sets in $W_i$ is weak order on $W_i$. For the indices $i$ which are in $I$, making $N_{W_i}(w_i)$ larger will make $B(F, \Phi')$ smaller; for the indices $i$ which are not in $I$, making $N_{W_i}(w_i)$ larger will make $B(F, \Phi')$ larger. This is why the $W_i$ factors are reversed for $i \in I$.
\end{proof}

\begin{corollary} \label{cor:PathComponent}
	Let $X$ and $Y$ be biclosed sets in $\Phi^+$. Then $X$ and $Y$ are commensurable if and only if there is a sequence $X = Z_0$, $Z_1$, $Z_2$, \dots, $Z_{\ell} = Y$ of biclosed sets with $Z_i$ and $Z_{i+1}$ differing by a single element.
\end{corollary}

\begin{proof}
	The reverse implication is immediate. To prove the forwards direction, let $X = B(F, \Phi',\allowbreak u)$ and let $Y = B(F, \Phi', v)$. Take chains $u = u_p > u_{p-1} > \cdots > u_1 > u_0  = e$ and $e = v_0 < v_1 < \cdots < v_{q-1} < v_q = v$ in $W_F$. 
	Then $B(F, \Phi', u_q)$, $B(F, \Phi', u_{q-1} )$, \dots, $B(F, \Phi', u_1)$, $B(F, \Phi', u_0)=B(F, \Phi') = B(F, \Phi', v_0)$, $B(F, \Phi', v_1)$, \dots, $B(F, \Phi', v_{q-1})$, $B(F, \Phi', v_q)$ is a sequence of the sort we seek.
\end{proof}

\begin{remark}\label{rem:pathcomponent}
	\Cref{cor:PathComponent} implies that commensurable biclosed sets are in the same connected component of the Hasse diagram of the poset of biclosed sets. In fact, commensurability classes are exactly the connected components of the Hasse diagram.
	To see this, we would need to know that whenever $X \subset Y$ is a cover relation in the poset of biclosed sets, then $X$ and $Y$ should be commensurable. An even stronger statement holds: if $X\subset Y$ is a cover relation, then $|Y\setminus X|=1$. This has been shown for affine Coxeter groups by Matthew Dyer, in an unpublished work \cite{Dyerpreprint}. 
	We give a different proof in \cite{Barkley2023}.
\end{remark}

\section{Separable and weakly separable sets} \label{sec:separable}

In this section, we determine which biclosed sets are separable and weakly separable. In part, we study these other notions in order to clarify the relationship between biclosure and other concepts in the literature. Also, in the course of this study, we prove \Cref{thm:rank3separable}, which will play a key role in 
our following paper~\cite{Barkley2023}.

\subsection{Separability and weak separability} Let $V$ be a real vector space and let $X$ be a set of nonzero vectors in $V$.
Let $B$ be a subset of $X$.
If $X$ is finite, then the hyperplane separation theorem states that the following are equivalent:
\begin{enumerate}
	\item There is a linear functional $\theta \in V^{\ast}$ such that $\langle \theta, \beta \rangle < 0$ for $\beta \in B$ and $\langle \theta, \beta \rangle > 0$ for $\beta \in X \setminus B$. \label{separableHyperplane}
	\item There do \textbf{not} exist $\beta_1$, $\beta_2$, \dots, $\beta_j \in B$ and $\gamma_1$, $\gamma_2$, \dots, $\gamma_k \in X \setminus B$ and positive scalars $b_1$, $b_2$, \dots, $b_j$, $c_1$, $c_2$, \dots, $c_k$ such that $\sum b_i \beta_i = \sum c_i \gamma_i$; here we permit one of $j$ and $k$ to be $0$ but not $j=k=0$.
	\label{separableCircuit}
\end{enumerate}

If $X$ is infinite, these conditions no longer coincide, though Condition~\ref{separableHyperplane} still clearly implies Condition~\ref{separableCircuit}. We define $B$ to be \newword{separable in $X$} if Condition~\ref{separableHyperplane} holds and \newword{weakly separable in $X$} if  Condition~\ref{separableCircuit} holds. We will omit the ``in $X$'' when $X$ is clear from context.

Here is an alternate description of weak separability which resembles the separating hyperplane condition: For a vector $\vec{x}=(x_1, x_2, \ldots, x_k) \in \RR^k$, define $\vec{x} \succ 0$ if there is an index $r$ such that $x_1 = x_2 = \cdots = x_{r-1} =0$ and $x_r>0$; define $\vec{x} \prec 0$ analogously. 

\begin{lemma}
	The set $B \subset X$ is weakly separable if and only if there exists a linear map $\theta : V \to \RR^k$ such that $\theta(\beta) \prec 0$ for $\beta \in B$ and $\theta(\beta) \succ 0$ for $\beta \in X \setminus B$. If we want, we may further assume that $\theta$ is an isomorphism of vector spaces.
\end{lemma}

\begin{proof}
	First, let $\theta : V \to \RR^k$ be any linear map which is nonzero on $X$. It is easy to check that $\{ \beta : \theta(\beta) \prec 0 \}$ is a weakly separable set.
	
	Now, suppose that $B$ is weakly separable in $X$. Fix an auxilliary inner product $( \ , \ )$ on $V$. 
	For each $\beta \in X$, let $H_{\beta}$ be the closed hemisphere
	\[ H_{\beta} = {\big\{} \theta \in V : (\theta, \theta)=1 \ \text{and} \ (\theta, \beta) \leq 0 \ \text{if}\  \beta \in B,\ (\theta, \beta) \geq 0 \ \text{if}\  \beta \in X \setminus B . {\big\}} \]
	For any finite subset $Y$ of $X$, the set $B \cap Y$ is weakly separable, hence separable, in $Y$, which means that $\bigcap_{\beta \in Y} H_{\beta}$ is nonempty. So the $H_{\beta}$ are compact subsets of the unit sphere in $V$ for which any finite subset has nonempty intersection, and therefore $\bigcap_{\beta \in X} H_{\beta}$ is nonempty. Let $\theta_1 \in\bigcap_{\beta \in X} H_{\beta}$. We will take $\theta_1$ to be the first component of the map $\theta$.
	
	If $(\theta_1, \beta)$ is nonzero for all $\beta \in X$, we are done. Otherwise, we can reduce to considering the set $B \cap \theta_1^{\perp}$ inside $X \cap \theta_1^{\perp}$. Since the vector space $\theta_1^{\perp}$ is lower dimensional than $V$, we are done by induction on dimension.
	
	Finally, we note that our proof gives a linearly independent (even orthonormal!) set $(\theta_1, \theta_2,\allowbreak \ldots, \theta_k)$ of vectors in $V$, and it is harmless to complete it to a basis, so we may assume that $\theta : V \to \RR^k$ is an isomorphism.
\end{proof}

\subsection{Separable biclosed sets}
We now specialize to the setting of the rest of the paper, considering separability and weak separability of biclosed sets in $\Phi^+$, for $\Phi^+$ a root system of affine type. 
\begin{theorem} \label{thm:separableCriterion}
	The biclosed set $B(F, \Phi', w)$ is separable if and only if either $F$ is either a maximal face of the $W_0$-Coxeter fan, or else $F = \{ 0 \}$; in the latter case, $B$ is either of the form $N(w)$ or $\Phi^+ \setminus N(w)$. 
\end{theorem}

\begin{proof}
	Recall that separable sets are induced by vectors $\theta$ in $V^{\ast}$ such that $\langle \theta, \beta \rangle \neq 0$ for all $\beta \in \Phi^+$. The points $\theta$ in $V^{\ast}$ which do not lie on any of the hyperplanes $\beta^{\perp}$ fall into three classes:
	\begin{enumerate}
		\item Points in the \newword{Tits cone} (meaning that $\langle \theta, \delta \rangle > 0$), which also lie in the complement of the $W$-Coxeter hyperplane arrangement.
		\item Points in the \newword{negative Tits cone} (meaning that $\langle \theta, \delta \rangle < 0$), which also lie in the complement of the $W$-Coxeter hyperplane arrangement.
		\item  Points in the boundary of the Tits cone (so $\langle \theta, \delta \rangle = 0$) which also lie in the complement of the $W_0$-Coxeter hyperplane arrangement.
	\end{enumerate}
	
	In the first case, we know that the regions of the complement of the $W$-Coxeter hyperplane arrangement in the Tits cone are indexed by $w \in W$, and the corresponding biclosed set is the inversion set $N(w) = B(\{ 0 \},\ \emptyset,\ w)$.
	In the second case, each region in the negative Tits cone is the negation of a region in the positive Tits cone, so the separable sets coming from $\theta$ in the negative Tits cone are of the form $\Phi^+ \setminus N(w) = B(\{ 0 \},\ \Phi^+,\ w)$.
	In the third case, let $F$ be the corresponding face of the $W_0$-Coxeter arrangement. The condition that $F$ does not lie in any $\beta^{\perp}$ means that $F$ is maximal. In this case, $W_F$ is the trivial group, so the data of $\Phi'$ and $w$ is trivial, and $B$ is the unique biclosed set for this $F$.
\end{proof}

\subsection{Weakly separable biclosed sets}
Our next result is 
\begin{theorem} \label{thm:weaklySeparableCriterion}
	The biclosed set $B(F, \Phi', w)$ is weakly separable if and only if either $\Phi' = \emptyset$ or $\Phi' = \Phi_F$. 
\end{theorem}

\begin{remark}  \label{InfiniteWords}
	If $\Phi' = \emptyset$ then $B$ is the inversions of a finite or infinite reduced word in the sense of~\cite{LP} and $B$ is biclosed in $\Phi^+ \cup \{\delta \}$. If $\Phi' = \Phi_F$ then  $\Phi^+ \setminus B$ is the inversions of a finite or infinite reduced word and  $B \cup \{ \delta \}$ is biclosed in $\Phi^+ \cup \{\delta \}$.
	If $F$ is a maximal face of the $W_0$-Coxeter arrangement, then $\Phi_F = \emptyset$, so we have $\Phi' = \emptyset = \Phi_F$ in this case and both $B$ and  $\Phi^+ \setminus B$ are the inversions of finite or infinite reduced words. \end{remark}

\begin{remark}
	Our description of weakly separable biclosed sets can largely be deduced from results of Wang, and of Hohlweg and Labb\'{e}, on inversions of infinite reduced words.
	(Throughout this remark, we use ``infinite reduced word'' as a shorthand for ``finite or infinite reduced word''.)
	The description of which biclosed sets correspond to infinite reduced words, as in Remark~\ref{InfiniteWords}, can also be found in Wang's paper \cite[Theorem~3.12]{Wang2018}; this relies on Dyer's unpublished preprint~\cite{Dyerpreprint} in the form of \cite[Theorem~1.13]{Wang2018}.
	Hohlweg and Labb\'{e}, \cite[Proposition~4.2(i)]{Hohlweg2016} shows that inversion sets of infinite reduced words are weakly separable.
	Hohlweg and Labb\'{e}'s result relies on a result of Cellini and Papi~\cite[Theorem~3.12]{Cellini1998} which relates infinite reduced words to a version of biclosed sets in an enlarged root system containing the imaginary root $\delta$.
\end{remark}

Theorems~\ref{thm:separableCriterion} and~\ref{thm:weaklySeparableCriterion} together prove Theorem~\ref{thm:SeparableMain}.

\begin{proof}[Proof of Theorem~\ref{thm:weaklySeparableCriterion}]
	First, suppose that $\Phi'$ is equal to neither $\emptyset$ nor $\Phi_F$. Let $P$ be a indecomposable component of $\Phi'$ and let $Q$ be a indecomposable component of $\Phi_F \setminus \Phi'$. Then $\delta$ is in the positive span of both $P$ and $Q$, so $B$ is not weakly separable.
	
	Now, suppose that $\Phi' = \emptyset$. In this case, we must show that $B(F, \emptyset, w)$ is separable.
	First, let $\theta_1$ be in the relative interior of $F$. Then $\langle \theta_1, \beta \rangle < 0$ for $\beta \in B(F, \emptyset, e)$, and  $\langle \theta_1, \beta \rangle = 0$ for $\beta \in \Phi_F$, and  $\langle \theta_1, \beta \rangle > 0$ for $\beta \in \Phi^+ \setminus (B(F, \emptyset, e) \cup \Phi_F)$. 
	It remains to find $\theta_2$ such that, when restricted to $\Phi^+_F$, $\langle \theta_2,\ \rangle$ is negative precisely for the inversions of $w$.
	
	Let $\Phi_0^1$, $\Phi_0^2$, \dots, $\Phi_0^r$ be the components of $\Phi_0^F$ and let $\Phi_i$ be the preimage of $\Phi_0^i$ in $\Phi$, so that the $\Phi_i$ are the components of $\Phi_F$. 
	Each of the $\Phi_i$'s is an affine root system, all of which share the same imaginary root $\delta$. 
	If we write $V_i = \text{Span}(\Phi_i)$ and $V_F = \text{Span}(\Phi_F)$, then the kernel of $\bigoplus V_i \longrightarrow V_F$ is spanned by the differences $\delta_i - \delta_j$, where $\delta_i$ is the copy of $\delta$ in the $i$-th summand. 
	We have $W_F = \prod W_i$, where the $W_i$ are the affine groups associated with the $\Phi_i$; let $w = (w_1, w_2, \ldots, w_r)$ in this product decomposition.
	For each $w_i$, we can find a $\theta_{2i} \in V_i^{\ast}$ such that $\langle \theta_{2i}, \ \rangle$ is negative on $N(w_i)$; moreover, $\theta_{2i}$ is in the Tits cone, so $\langle \theta_{2i}, \delta_i \rangle > 0$. Normalize the $\theta_{2i}$'s so that $\langle \theta_{2i}, \delta_i \rangle = 1$. Then $\sum \langle \theta_{2i},\ \rangle$ vanishes on the kernel of $\bigoplus V_i \longrightarrow V_F$, so there is some $\theta_2 \in V_F^{\ast}$ which restricts to $\theta_{2i}$ on each $V_i$. Extending this $\theta_2$ to a linear functional on all of $V$, we have the functional we seek. 
	
	We have now verified, using the map $(\theta_1, \theta_2) : V \to \RR^2$, that $B(F, \emptyset, w)$ is weakly separable. 
	The verification for $B(F, \Phi', w)$ is similar.
\end{proof}

\begin{cor}\label{thm:rank3separable}
	If $\Phi$ is an affine root system of rank $\leq 3$, then every biclosed set is weakly separable.
\end{cor}

\begin{proof}
	In this case, $\Phi_0^F$ must be either the whole Coxeter diagram $\Phi_0$, or a single vertex of it, or the empty set. In any of these cases, $\Phi_0^F$ has at most one component, so we always have either $\Phi'=\emptyset$ or $\Phi' = \Phi_F$ (or both).
\end{proof}

\section{Combinatorial descriptions}\label{sec:combinatorial}
In this section we apply our previous results to give a combinatorial description of 
the commensurability classes and the biclosed sets for types $A_{\infty}$, $\tAn$, $\tC_n$, $\tB_n$ and $\tD_n$. Additionally, we give combinatorial proofs of the lattice property for types $A_{\infty}$, $\tAn$ and $ \tC_n$.

The results in this section, though not their proofs, can be understood independently of Sections~\ref{sec:AffineModels} and \ref{sec:separable}.

We will use the combinatorial descriptions of the classical affine Coxeter groups that are described in the book by Bj\"orner and Brenti \cite{Bjorner2005}. 
This source does not emphasize root systems, so we will include descriptions of the root systems as well.


\subsection{The Coxeter group \texorpdfstring{$A_\infty$}{A\_infinity}}
We begin with a Coxeter group of infinite rank, the (two-sided) infinite symmetric group $A_\infty$.

\begin{definition}
	The Coxeter group $A_\infty$ is the group of permutations $\pi: \ZZ\rightarrow \ZZ$ which fix all but finitely many points. The simple generators are the simple transpositions $s_i = (i, i+1)$ for $i\in \ZZ$. 
\end{definition}

The Coxeter diagram of $A_{\infty}$ is
\begin{figure}[H]
	\centering
	\begin{dynkinDiagram}[scale=1.5,indefinite edge ratio = 2, labels={,-3,-2,-1,...,3,}]{A}{*.*******.*}
		\fill[white] (root 1) circle (.24cm) (root 9) circle (.24cm);
	\end{dynkinDiagram}.
\end{figure}

The reflections of $A_{\infty}$ are the transpositions $(i,j)$ for $i<j$. 
Letting $\RR^{\infty}$ be the vector space with basis $e_i$ for $i \in \ZZ$, we can take $V$ to be the span of the vectors $\beta_{ij} \coloneqq  e_j - e_i$. The positive roots are the $\beta_{ij}$ for $i<j$, and the simple roots are $\alpha_i\coloneqq  \beta_{i,i+1}$ for $i\in \ZZ$.

The rank 2 subsystems of $A_{\infty}$ are:
\[ 	\begin{array}{ll}
	\beta_{ij} \rightarrow \beta_{ik} \rightarrow \beta_{jk} & \text{for $i<j<k$} \\
	\beta_{ij} \rightarrow \beta_{k \ell} & \text{for $i,j,k,\ell$ pairwise distinct.}\\
	
\end{array} \]

\begin{corollary}
	Let $\overline{\ \cdot\ }$ be the closure operator for $A_\infty$. Then $\overline{\ \cdot\ }$ is the weakest closure operator on subsets of $\Phi^+$ which satisfies 
	\[ \overline{\{ \beta_{ij}, \beta_{jk} \}} = \{\beta_{ij}, \beta_{ik}, \beta_{jk} \} \]
	for all $i<j<k$. 
\end{corollary}
Thus, if we think of a set $U$ of positive roots as defining a relation $R$ on $\ZZ$, where $i R j$ if $\beta_{ij} \in U$, then $\overline{U}$ corresponds to taking the transitive closure of this relation.

We can use this idea to show that the poset of biclosed sets is a lattice.
We start with a lemma.

\begin{lemma} \label{lem:KeyClosureLemma}
	If $U\subseteq \Phi^+$ is a coclosed set, then the closure $\overline{U}$ is biclosed.
	If $K \subseteq \Phi^+$ is a closed set, then the interior $K^{\circ}$ is closed.
\end{lemma}

For legibility, in this proof we label the elements of $U$ by the corresponding transpositions $(i,j)$ in $T$, so we write $(i,j)$ instead of $\beta_{ij}$.

\begin{proof}
	Let $U$ be a coclosed set. By the corollary, $\overline{U}$ is the set of all pairs $(a,c)$ with $a<c$ such that there exists a sequence
	\[ a = b_0 < b_1 < \ldots < b_{r-1} < b_r = c\]
	\[ (b_{i},b_{i+1}) \in U \text{ for all $i$} .\]
	Let $a<b<c$ such that $(a,c)\in \overline{U}$. To show that $\overline{U}$ is coclosed, we must show that either $(a,b)\in \overline{U}$ or $(b,c)\in \overline{U}$.
	Since $(a,c)\in \overline{U}$, we have a chain $b_0 < \ldots < b_r$ as above. If $b=b_i$ for some $i$, then both $(a,b)$ and $(b,c)$ are in $\overline{U}$ and we are done. Otherwise there is some $i$ for which $b_i < b < b_{i+1}$. However, we know that $U$ is a coclosed set and that 
	\[ (b_i, b) \rightarrow (b_i, b_{i+1}) \rightarrow (b, b_{i+1}) \]
	is a rank 2 system. Since $(b_i,b_{i+1})$ is in $U$, it follows that either $(b_i,b) \in U$ or that $(b,b_{i+1})\in U$. In the first case, $(a,b)\in \overline{U}$, and in the second case, $(b,c)\in \overline{U}$. Thus $\overline{U}$ is coclosed. Since it is closed by definition, it is also biclosed.
	
	The case of a closed set follows since $K^{\circ} = \pos \setminus \left( \overline{\pos \setminus K} \right)$.
\end{proof}

\begin{theorem} \label{AInftyLattice}
	The poset $\bic(A_\infty)$ is a complete lattice.
\end{theorem}

\begin{proof}
	Let $\cX$ be any subset of $\bic(A_\infty)$; we will show that $\cX$ has a least upper bound.

	Put $U = \bigcup_{I \in \cX} I$. Since $U$ is a union of coclosed sets, it is coclosed. Put $V = \overline{U}$. Since $V$ is a closure, it is closed; by \Cref{lem:KeyClosureLemma} it is also coclosed. So $V$ is a biclosed set. We claim that it is the least upper bound of $\cX$.
	
	Clearly, for each $I \in \cX$, we have $I \subseteq U \subseteq V$, so $V$ is an upper bound for $\cX$.
	Let $Y$ be any other upper bound for $\cX$. Then, for each $I \in \cX$, we have $I \subseteq Y$, so $U \subseteq Y$. 
	Since $Y$ is closed, we have $V = \overline{U} \subseteq Y$ as well. So $V$ is below every upper bound for $\cX$.
	
	The case of a greatest lower bound is analogous, using $\left( \bigcap_{I \in \cX} I \right)^{\circ}$.
\end{proof}

It will turn out that biclosed sets in $A_\infty$ are in correspondence with total orders of the integers. To state this precisely, we introduce the following notations. 
\begin{definition}
	Let $R$ be a binary relation on a set $X$. We say that $R$ is \newword{trichotomous} if, for all $x$, $y \in X$, exactly one of $xRy$, $yRx$, and $x = y$ holds.
	A \newword{total order} is a transitive trichotomous relation. Let $\tot(\ZZ)$ denote the collection of total orderings of $\ZZ$. For $x\in \tot(\ZZ)$, we write $\prec_x$ for the corresponding total order. We always write $<$ for the standard total order 
	$$\cdots < -2 < -1 < 0 < 1 < 2 < \cdots .$$
	We define an \newword{inversion} of $\prec_x$ to be a root $\beta_{ij}$ such that $i<j$ and $i \succ_x j$, and let $N(\prec_x)$ denote the set of such inversions.  
\end{definition}

Using this, we have the following neat characterization of biclosed sets.
\begin{theorem}
	Let $B \subseteq \Phi^+$ be a set of positive roots for $A_\infty$. Then $B$ is biclosed if and only if $B=N(\prec_x)$ for some total order $\prec_x\in \tot(\ZZ)$. 
\end{theorem}
\begin{proof}
	Let $R$ be a trichotomous relation on $\ZZ$. Then $R$ is determined by the pairs $a<b$ for which $bRa$. Such a pair determines a positive root $\beta_{ij}$ of $A_\infty$, which we consider to be an inversion of $R$. Let $B_R$ be the set of inversions associated to $R$. 
	
	Then $B_R$ is closed if and only if there are no $a<b<c$ with $c R b R a R c$, and $B_R$ is coclosed if and only if there are no $a<b<c$ with $a R b R c R a$.
	So $B_R$ is biclosed if and only if $R$ is transitive.
	Since a trichotomous relation is a total order iff it is transitive, we thus have that $R$ is a total order $\prec_x$ if and only if $B$ is biclosed.
\end{proof}

Thus the poset of biclosed sets on $A_\infty$ is isomorphic to the poset of inversion sets of total orders on $\ZZ$, which we may naturally interpret as the ``weak order'' on $\tot(\ZZ)$. 



%
%
%
%
%

\subsection{The affine group \texorpdfstring{$\An$}{\~A\_\{n-1\}}}
\begin{definition}
	Let $n\geq 1$. The \newword{affine symmetric group} $\An$ is the group of permutations $\pi:\ZZ\rightarrow \ZZ$ which are equivariant under translation by $n$:
	\[ \pi(i+n) = \pi(i) + n \quad \forall i\in \ZZ \]
	which satisfy the normalization condition
	\[ \sum_{i=0}^{n-1}\pi(i) = \sum_{i=0}^{n-1}i .\]
\end{definition}
When $n\geq 3$, the Coxeter diagram of $\tAn$ is shown below:

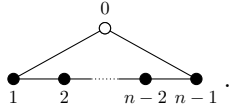
\begin{center}
	\dynkin[scale=1.5,extended,Coxeter,labels={0,1,2,n-2,n-1},edge length = .45cm]{A}{}.
\end{center}

The reflections of $\tAn$ are
\[ t^{\tA}_{ij} \coloneqq  \prod_{r \in \ZZ} (i+rn, j+rn) \qquad i \not \equiv j \bmod n. \]
As we described earlier, the root system can be embedded in an $(n+1)$-dimensional vector space with spanning set $\te_i$ for $i \in \ZZ$ and relations $\te_{i+n} - \te_i = \te_{j+n} - \te_j$.
The corresponding positive roots are  $\beta_{ij} \coloneqq  \te_j - \te_i$ for  $i < j$ and $i \not\equiv j \bmod n$. The simple roots are $\alpha_i \coloneqq \beta_{i,i+1}$ for $0\leq i < n$. We also define the vector $\delta \coloneqq \te_n-\te_0$ for convenience later.

\begin{remark}
	The affine group $\tAn$ is sometimes only defined for $n \geq 3$, but we take the above definition for $\tA_{2-1}$ and $\tA_{1-1}$ as well. The Coxeter group $\tA_{1}$ has two generators, $s_0\coloneqq t^{\tA}_{01}$ and $s_1\coloneqq t^{\tA}_{12}$, with $m_{01} = \infty$. The Coxeter group $\tA_0$ is the trivial group, with no generators.
\end{remark}

We now describe what \Cref{Parametrization} says in the case of $\tAn$. 
The vector space $V_0^{\ast}$, is naturally coordinatized as $\RR^n / \RR(1,1,\ldots, 1)$, and we will write elements of $V_0^{\ast}$ as $(f_1, f_2, \ldots, f_n)$. 
The symmetric group $W_0 = S_n$ acts on these coordinates by permuting the $f_i$.

The reflecting hyperplanes $V_0^{\ast}$ are of the form $f_i = f_j$, so the faces of the Coxeter fan are in bijection with ordered set partitions (equivalently, total preorders) of $[n]$. We can determine the total preorder from the face $F$ as follows: if $f$ is a functional in the relative interior of $F$, then $j\geq i$ in the preorder if and only if $f(\te_j-\te_i)\geq 0$. In particular, $i$ and $j$ are in the same block of the ordered set partition if and only if $f(\beta_{ij})=0$. For example, with $n=4$, the ordered set partition $(\{1,3\}, \{2,4 \})$ corresponds to the face $\{ f_1 = f_3 \leq f_2 = f_4 \}$. 
So, to give a face $F$ is to give an ordered set partition $(B_1, B_2, \ldots, B_r)$ of $[n]$. 

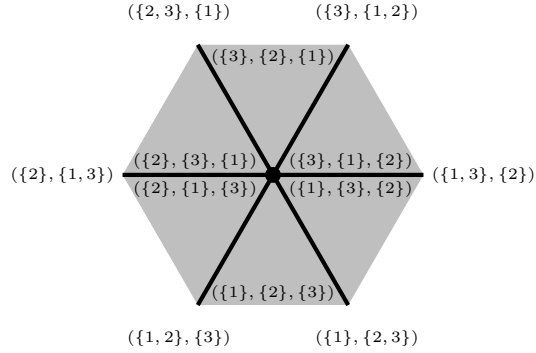
\begin{figure}[H]
	\centering
	\begin{tikzpicture}
		\fill[lightgray, opacity=.5] (0:2) -- (60:2) -- (120:2) -- (180:2) -- (240:2) -- (300:2) -- cycle;
		
		\draw[ultra thick] (0:2) -- (0:-2);
		\draw[ultra thick] (60:2) -- (60:-2);
		\draw[ultra thick] (120:2) -- (120:-2);
		
		\fill (0,0) circle (3pt);
		
		\draw[font=\tiny] node at (90:-1.57) {$(\{1\} , \{2\} , \{3\})$};
		\draw[font=\tiny] node at (90:1.57) {$(\{3\} , \{2\} , \{1\})$};
		\draw[font=\tiny] node[text centered] at (170:-1.05) {$(\{1\} , \{3\} , \{2\})$};
		\draw[font=\tiny] node[text centered] at (170:1.05) {$(\{2\} , \{3\} , \{1\})$};
		\draw[font=\tiny] node[text centered] at (10:-1.05) {$(\{2\} , \{1\} , \{3\})$};
		\draw[font=\tiny] node[text centered] at (10:1.05) {$(\{3\} , \{1\} , \{2\})$};
		
		\draw[font=\tiny] node at (60:-2.5) {$(\{1,2\} , \{3\})$};
		\draw[font=\tiny] node at (60:2.5) {$(\{3\} , \{1,2\})$};
		\draw[font=\tiny] node at (120:-2.5) {$(\{1\} , \{2,3\})$};
		\draw[font=\tiny] node at (120:2.5) {$(\{2,3\} , \{1\})$};
		\draw[font=\tiny] node at (0:-2.8) {$(\{2\} , \{1,3\})$};
		\draw[font=\tiny] node at (0:2.8) {$(\{1,3\} , \{2\})$};
	\end{tikzpicture}
	\caption{The $A_2$ Coxeter fan. Each face is labeled by an ordered partition of $\{1,2,3\}$. The origin is also a face, with label $(\{1,2,3\})$.}\label{fig:A2Fan}
\end{figure}

The root system $\Phi_F$ consists of the roots $\te_p - \te_q$ where $p \bmod n$ and $q \bmod n$ belong to the same block $B_i$ of the set partition.
Thus, the indecomposable components of $\Phi_F$ are in bijection with the nonsingleton blocks $B_i$. 
So, to give a pair $(F, \Phi')$ is to give an ordered set partition $(B_1, B_2, \ldots, B_r)$ of $[n]$ and a subset $R$ of the non-singleton blocks of this partition.
Set $n_i = |B_i|$.

The face subgroup $W_F$ is a product $\prod \tA_{n_i-1}$, where $\tA_{n_i-1}$ acts on the integers which reduce to elements of $B_i$ modulo $n$.

In summary, to give a triple $(F, \Phi', w)$ in type $\tAn$ corresponds to giving:
\begin{enumerate}
	\item An ordered set partition $(B_1, B_2, \ldots, B_r)$ of $[n]$.
	\item A subset $R$ of the nonsingleton blocks of this partition.
	\item For each non-singleton block $B_i$, an affine permutation $w_i$ in $\tA_{n_i-1}$. 
\end{enumerate}

We now explain how to translate this data into a total order on $\ZZ$, as described in \Cref{sec:NiceIntro}.
Let $p \bmod n$ and $q \bmod n$ be in the blocks $B_i$ and $B_j$.
If $i \neq j$, then we order $p \prec q$ if $i<j$ and order $p \succ q$ if $i>j$.
So each of the $B_i + n \ZZ$ forms an interval for our order, and we must describe the ordering of each $B_i + n \ZZ$.
If $B_i$ is a singleton, recall that we do not need to choose an order on $B_i + n \ZZ$.
Finally, we must describe how to order $B_i + n \ZZ$ if $B_i$ is not a singleton.
If $B_i \not\in R$, then we order $B_i + n \ZZ$ by the permutation $w_i$; if $B_i \in R$, then we order $B_i  + n \ZZ$ by the reverse of $w_i$.
We leave to the reader the routine check that the biclosed set $B(F, \Phi', w)$ is then $\{ \te_q - \te_p \in \Phi^+ : p \succ q \}$. 
This completes the proof of \Cref{AtildeModel}.

We will now show that we can use  \Cref{AtildeModel} to prove that biclosed sets in type $\tAn$ form a lattice.
Let $L_n$ be the set of translation invariant total orders on $\ZZ$, meaning total orders that obey $i \prec j$ if and only if $i+n \prec j+n$.
Since the set of all total orders of $\ZZ$ is identified with $\bic(A_{\infty})$, we embed $L_n$ into $\bic(A_{\infty})$; order $L_n$ by the order induced from $\bic(A_{\infty})$.

%
%

\begin{lemma}
	The poset $L_n$ is a complete sublattice of $\bic(A_{\infty})$ and, in particular, is a complete lattice.
\end{lemma}

\begin{proof}
	Let $\cX$ be a subset of $L_n$. The meet and join $\bigmeet \cX$ and $\bigjoin \cX$ exist in $\bic(A_{\infty})$; we must show that $\bigmeet \cX$ and $\bigjoin \cX$  are also translation invariant.
	
	Let the group $\ZZ$ act on the set of total orders by $a (\prec + k) b$ if and only if $a+kn \prec b+kn$. Then $L_n$ is the fixed points for this action.
	Since this action preserves the poset structure, it commutes with meet and join, so the meet and join of a collection of fixed point must also be fixed.
\end{proof}

We now construct maps $\pi: L_n \onto \bic(\tAn)$ and $\iota : \bic(\tAn) \into L_n$.
The map $\pi$ is simple: Given a translation invariant biclosed subset $X$ of $A_{\infty}$, we  put
\[ \pi(X) = \{ \beta_{ij} :  (i,j) \in X,\ i < j,\ i \not\equiv j \bmod n \} . \]

To define $\iota$, let $X \in \bic(\tAn)$. 
Then $X$ corresponds to a translation invariant total order on $\ZZ$ which is defined up to reversing intervals of the form $r+n \ZZ$.
Let $\iota(X)$ be the  translation invariant total order where we order each such interval in increasing order.

\begin{lemma} 
	Let $X$ be a translation invariant biclosed subset of $A_\infty$. Then $\pi(X)$ is a biclosed subset of $\tAn$. 
\end{lemma}

\begin{proof}
	It's enough to check that $\pi(X)$ is closed; taking complements then gives that $\pi(X)$ is coclosed. Write $\prec$ for the translation invariant total order associated to $X$. We can consider the closure criteria (1)-(3) from \Cref{sec:NiceIntro}. Let $i<j<k$. If $\beta_{ij},~\beta_{jk}\in \pi(X)$, then $(i,j)$ and $(j,k)$ are in $X$, so $i\succ j$ and $j\succ k$, hence $i\succ k$. If $i\not\equiv k\bmod n$, then $\beta_{ik}\in X$ and (1) is satisfied. If instead $i\equiv k\bmod n$, then it must be that $k\succ k+n\succ k+2n\succ\cdots$, since in a translation invariant total order, the elements of $\ZZ$ in a given residue class must be ordered in increasing or decreasing order. Hence $(i,k+\ell n)$ is in $X$ for all $\ell\geq 0$, so (2) is satisfied. If now we assume $\beta_{ij}$ and $\beta_{ik}\in \pi(X)$, with $j\equiv k\bmod n$, then $i\succ j$ and $i\succ k$. The $\prec$-interval between $j$ and $k$ will contain every integer of the form $j+\ell n$ with $j\leq j+\ell n \leq k$, again since residue classes appear in monotone order. Hence $i\succ j+\ell n$ for such choices of $\ell$. Therefore (3) is satisfied and $\pi(X)$ is closed.
\end{proof}

\begin{lemma} \label{lem:PiIotaMaps}
	The maps $\pi:L_n\to \bic(\tA_n)$ and $\iota:\bic(\tA_n)\to L_n$ are maps of posets satisfying $\pi(\iota(X))=X$ for any $X\in\bic(\tAn)$.
\end{lemma}

\begin{proof}
	It is clear that $\pi$ is a map of posets, since $L_n$ is ordered by containment of subsets of $\{ (i,j) : i<j \}$, and $\pi$ simply intersects these subsets with  $\{ (i,j) : i<j,\ i \not\equiv j \bmod n \}$.
	
	It is almost as clear that $\iota$ is a map of posets. The only issue is the following: Let $X$ and $Y \in \bic(\tA_n)$ with $X \subseteq Y$. We need to rule out the possibility that some residue class $r+n \ZZ$ is ordered in decreasing order in $\iota(X)$ but in increasing order in $\iota(Y)$. 
	Let $\prec_X$ and $\prec_Y$ be translation invariant total orders corresponding to $X$ and $Y$. 
	Since $r+n \ZZ$ is ordered in decreasing order in $\iota(X)$, the residue class $r+n \ZZ$ is not an interval for $\prec_X$, so we can find some $s \not \equiv r \bmod n$ and some $q>0$ such that $r < s < r+qn$ and $r \succ_X s \succ_X r+qn$. Then $(r,s)$ and $(s, r+qn)$ must be inverted in $Y$ as well, so $r \succ_Y r+qn$.
	
	Now let $X\in\bic(\tAn)$. That $\pi(\iota(X))=X$ follows from the fact that, if $p<q$ and $p\not\equiv q\bmod n$, then $(p,q)$ is an inversion of $\iota(X)$ if and only if $\beta_{pq}$ is an inversion of $X$.
\end{proof}

We now deduce:
\begin{theorem} \label{thm:TypeALattice}
	The poset $\bic(\tAn)$ is a complete lattice.
\end{theorem}

\begin{proof}
	Put $p = \iota \circ \pi$. 
	By \Cref{lem:PiIotaMaps}, $p$ is an idempotent map of posets $L_n \to L_n$, and $\bic(\tAn)$ is isomorphic to the poset $\Fix(p)$ of fixed points of $p$.
	
	Let $\cX \subseteq \Fix(p)$; we must show that $\cX$ has a join and a meet in $\Fix(p)$.
	We do the case of the join; the meet is analogous.
	
	The join $\bigvee \cX$ exists in $L_n$; call it $Z$. We claim that $p(Z)$ is the join of $\cX$ in $\Fix(p)$. 
	For any $I \in \cX$, we have $I \subseteq Z$ and thus, since $p$ is order preserving, we have $I = p(I) \subseteq p(Z)$. 
	So $p(Z)$ is an upper bound for $\cX$.
	
	Let $Y \in \Fix(p)$ be any other upper bound for $\cX$. 
	Since $Z$ is the join of $\cX$ in $L_n$, we have $Z \subseteq Y$ and thus $p(Z) \subseteq p(Y) = Y$. So $p(Z)$ is the least upper bound for $\cX$, as desired.
\end{proof}

\begin{eg} 
	We close with an extended example of computing a join in $\tA_{4-1}$, both in terms of roots and in terms of total orders. 
	Specifically, we will compute $(s_0 s_1) \join (s_2 s_3)$. 
	The inversion sets of $s_0 s_1$ and $s_2 s_3$ are $\{ \alpha_0, \alpha_0 + \alpha_1 \}$ and $\{ \alpha_2, \alpha_2 + \alpha_3 \}$ respectively. 
	According to \Cref{AtildeLattice}, the join of these inversion sets should be $\overline{\{\alpha_0, \alpha_0+\alpha_1, \alpha_2, \alpha_2+\alpha_3 \}}$. 
	
	We claim that this closure is 
	\[B\coloneqq    \bigcup_{k=0}^{\infty} \{ \alpha_0 + k \delta,\ \alpha_0+\alpha_1+k \delta,\  \alpha_0+\alpha_1+\alpha_2 + k \delta,\  \alpha_2 + k \delta, \alpha_2+\alpha_3+k \delta,\ \ \alpha_2+\alpha_3+\alpha_0 + k \delta \} .\]
	We sketch the proof: First of all, $\alpha_0+\alpha_1$ and $\alpha_2+\alpha_3$  are the base of a $\tA_1$ subsystem, so the closure contains each root of the form $\alpha_0+\alpha_1+k \delta$ and $\alpha_2+\alpha_3+k \delta$. 
	Then, using the identities $\alpha_0+k \delta = (\alpha_0+\alpha_1+(k-1) \delta)+(\alpha_2+\alpha_3+\alpha_0)$, $\alpha_2+k \delta = (\alpha_2+\alpha_3+(k-1) \delta)+(\alpha_0+\alpha_1+\alpha_2)$, $\alpha_0+\alpha_1+\alpha_2 + k \delta = (\alpha_0+\alpha_1 + k \delta) + \alpha_2$ and $\alpha_2+\alpha_3+\alpha_0 + k \delta = (\alpha_2+\alpha_3 + k \delta) + \alpha_0$, we deduce that the other elements of $B$ are in the closure.
	An easy case check shows that $B$ is closed, so $B$ is the closure.
	A further check verifies that $B$ is biclosed, as promised in  \Cref{AtildeLattice}.
	In terms of our classification of biclosed sets, the ordered set partition is $(\{ 1,3 \},\ \{ 2, 4 \})$, with $\Phi'$ corresponding to the block $\{ 2,4 \}$.
	
	We observe that $B$ contains $\alpha_0+\alpha_1$ and $\alpha_2+\alpha_3$, but not $\alpha_1+\alpha_2$ and $\alpha_3+\alpha_0$, even though $(\alpha_0+\alpha_1)+(\alpha_2+\alpha_3) = (\alpha_1+\alpha_2)+(\alpha_3+\alpha_0)$. So we cannot find a $\theta \in V^{\ast}$ which pairs positively (or nonnegatively) with the elements of $B$ and negatively with the elements of $\Phi^+\setminus B$.
	In other words, this biclosed set is not weakly separable. This matches \Cref{thm:weaklySeparableCriterion}, since $\Phi_F$ has two components, corresponding to the two blocks of  $(\{ 1,3 \},\ \{ 2, 4 \})$, and $\Phi'$ is one of the two components.
	
	In terms of total orders, we have
	\[ \begin{array}{l@{~\longleftrightarrow\ \cdots \prec}c@{\ \prec \ }c@{\ \prec \ }c@{\ \prec \ }c@{\ \prec \ }c@{\ \prec \ }c@{\ \prec \ }c@{\ \prec \ }c@{\prec \cdots}}
		\{ \alpha_0, \alpha_0 + \alpha_1 \}  &   1  &  2  &  0  &  3  &  5  &  6  &  4  &  7  \\
		\{ \alpha_2, \alpha_2 + \alpha_3 \} &     -2  &  1  &  3  &  4  &  2  &  5  &  7  &  8  \\
	\end{array}\]
	The join $B$ corresponds to the order
	\[ \cdots 1 \prec 3 \prec 5 \prec 7 \prec \cdots \prec 8 \prec 6 \prec 4 \prec 2 \prec 0 \prec \cdots .\]
\end{eg}

\subsection{The affine group \texorpdfstring{$\tC_n$}{\~C\_n}} \label{sec:Cmodel}

In this section and in \Cref{sec:Bmodel,sec:Dmodel}, set $N = 2n+1$ and adopt the convention that the variables $i$, $j$, $k$ and $\ell$ denote integers which  are not divisible by $N$.
Much of the notation we will introduce below will be useful not only for $\tC_n$ but also for $\tB_n$ and $\tD_n$.

\begin{definition}
	Let $n\geq 1$. The \newword{signed affine symmetric group} $\tC_n$ is the subgroup of $\tS_N = \tA_{N-1}$ consisting of permutations $f$ obeying $f(x) = - f(-x)$.
\end{definition}

\begin{remark}
	We note that our conditions  $f(x) = - f(-x)$ and $f(x+N) = f(x) + N$ force $f(x) = x$ for any $x \equiv 0 \bmod N$. 
	The elements of $N \ZZ$ thus play no important role and we could discard them and then renumber $\ZZ \setminus (N \ZZ)$ by the elements of $\ZZ$.
	If one does this, we obtain the subgroup of permutations in $\tS_{N-1}$ which obey $f(1-x) = 1-f(x)$; this group is also isomorphic to $\tC_n$. 
	However, we have decided to follow the conventions of~\cite{Bjorner2005}, which uses the periodicity we have imposed here.
\end{remark}

\begin{remark}
	Usually, $\tC_n$ is only defined for $n \geq 2$, but we will use the above definition for $n=1$ as well. In that case, $\tC_1$ is the subgroup of $\tA_{3-1}$ generated by $t^{\tA}_{(-1)1}$ and $t^{\tA}_{12}$; as an abstract Coxeter group, $\tC_1 \cong \tA_1$.
\end{remark}

The Coxeter diagram for $\tC_n$, when $n\geq 2$, is:
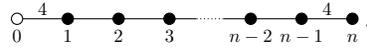
\begin{center}
	\dynkin[scale=1.5,extended,Coxeter,labels={0,1,2,3,n-2,n-1,n},edge length = .45cm]{C}{}.
\end{center}

We now describe the reflections in $\tC_n$, they are
\[ t^{\tC}_{ij} \coloneqq  t^{\tA}_{ij} t^{\tA}_{(-i)(-j)} \qquad i \not\equiv \pm j \bmod N \]
\[ t^{\tC}_{ij} = t^{\tA}_{ij} \qquad i+j \equiv 0 \bmod N. \]
(We recall our standing convention that $i$ and $j$ are never $0 \bmod N$.)
For $i < j$ with $i\not\equiv j \bmod N$, the reflection $t_{ij}$ is an inversion of $f$ if and only if $f(j) < f(i)$.

Let $V$ be the $(n+1)$-dimensional real vector space with spanning set $\te_x$, indexed by the integers, modulo the relations $\te_{x+N} = \te_{x} + \te_N$ and $\te_{-x} = -\te_{x}$. 
(The second condition implies that $\te_0=0$.)
Then the root system $\Phi$ can be identified with the vectors $\te_j - \te_i$ for $i \not\equiv j \bmod N$ in this vector space.
The positive root corresponding to $t^{\tC}_{ij}$ is $\beta_{ij}\coloneqq \te_j - \te_i$. The simple roots are $\alpha_i\coloneqq\beta_{i,i+1}$ for $i=1,\ldots,n$ and $\alpha_0\coloneqq \beta_{-1,1}$. We also define the vector $\delta\coloneqq \te_{N}-\te_0 = \te_N$ for convenience later.

\begin{remark}
	We have $\beta_{-i,i} = \te_{i} - \te_{-i} = 2 \te_i$. This is a ``long root''. We will not need this level of detail, but we include it to confirm that this model matches the usual crystallographic root system. 
\end{remark}

We will want to work with faces of the Coxeter fan of type $C_n$. 
These are indexed by \newword{signed ordered set partitions}.
A signed ordered set partitions is a total preorder on $\{ 0, \pm 1, \pm 2, \ldots, \pm n \}$ obeying $i \preceq j$ if and only if $(-j) \preceq (-i)$.
We note that such a partition always has an odd number of parts,  $p_{-r}$, \dots, $p_{-2}$, $p_{-1}$, $p_0$, $p_1$, $p_2$, \dots, $p_r$, with $0 \in p_0$. 
We call $p_0$ the central block. 

\begin{theorem}\label{thm:typeCorders}
	Let $\prec$ be a total order on $\ZZ$ which is invariant for translation (meaning $i \prec j$ implies $(i+N) \prec (j+N)$) and for negation (meaning that $i \prec j$ implies $-j \prec -i$).
	Then $\{ \te_j - \te_i : i<j,\ i \succ j,\ i \not\equiv j \bmod N \}$ is a biclosed set of $\tC_n$, and all biclosed sets for $\tC_n$ are of this form.
	Two total orders on $\ZZ$ give the same biclosed set if and only they are the same up to reversal on intervals of the form $r+N \ZZ$.
\end{theorem}

\begin{proof}[Proof sketch]
	Let $F$ be a face of the $C_n$ Coxeter fan corresponding to the signed ordered set partition $p_{-r}$, \dots, $p_{-2}$, $p_{-1}$, $p_0$, $p_1$, $p_2$, \dots, $p_r$. The corresponding face subgroup is $\tC_{(|p_0|-1)/2} \times \prod_{i=1}^r \tA_{|p_i|-1}$, where we take the trivial factor for any singleton $p_i$, including the case $p_0 = \{ 0 \}$. Each of the nontrivial factors is a connected Coxeter group, so choosing $\Phi'$ is just choosing a subset of the non-singleton blocks, with the condition that if we choose $p_i$ then we must also choose $p_{-i}$. 
	
	The biclosed set $B(F, \Phi')$ corresponds to the total order on $\ZZ$ where we take 
	\[(p_{-r} + N \ZZ) \prec \cdots \prec  (p_{-1} + N \ZZ) \prec (p_0 + N \ZZ) \prec (p_1 + N \ZZ)  \prec \cdots \prec (p_r + N \ZZ)\]
	and order $p_i + N \ZZ$ forwards if $p_i$ is not in $\Phi'$ and backwards if $p_i$ is in $\Phi'$.
	As in Type $\tA$, if $p_i$ is a singleton then it doesn't matter if we order $p_i+N \ZZ$ forwards or backwards.
	
	Then $B(F, \Phi', w)$ corresponds to acting on the order on each $p_i+N \ZZ$ by $w$.
\end{proof}

\begin{eg}
	Define an auxiliary set of roots $B_\infty =$
	\[ \{ \beta_{-1,1},\beta_{-1,2},\beta_{-1,3},\beta_{-1,4},-\beta_{-4,4},-\beta_{-3,4},-\beta_{-2,4},-\beta_{-1,4},-\beta_{1,2},-\beta_{1,3},-\beta_{1,4},-\beta_{2,4},-\beta_{3,4} \}  \]
	and consider the set of positive roots
	\[ B = \{\beta_{-3,3},\beta_{-2,2},\beta_{-2,3},\beta_{-1,1},\beta_{-1,2},\beta_{-1,3},\beta_{2,3} \}\cup  \bigcup_{k\geq 1}\{\beta+k\delta \mid \beta\in B_\infty \}.\]
	One can verify that $B$ is a biclosed set. Let's determine the parameters $(F,\Phi',w)$ so that $B = B(F,\Phi',w)$. We first compute $B_\infty$, which coincides with the auxiliary set shown above. From $B_\infty$ we can extract the data of a face $F$ of the $C_n$-Coxeter fan, and a subsystem $\Phi'$ of $\Phi_F$. The total preorder associated to $F$ satisfies $i\preceq j$ if and only if $\beta_{i,j}\in B_\infty$. We deduce that the signed ordered set partition associated to $F$ is	
	\[(\{-4,1\}, \{-3,-2,0,2,3\},\{-1,4\}).\]
	Then $\Phi_F$ consists of roots $\pm\beta_{i,j}+k\delta$ with $i,j$ sharing a block in this partition, so 
	\begin{align*} \Phi_F = \Phi_1\cup\Phi_2=&\{ \pm\beta_{-3,3}+k\delta, \pm\beta_{-2,2}+k\delta,\pm\beta_{-2,3}+k\delta, \pm\beta_{2,3}+k\delta \mid k\in \ZZ \}\\ \cup &\{ \pm\beta_{-1,4}+k\delta \mid k\in \ZZ \},
	\end{align*}
	and this is the decomposition of $\Phi_F$ into its two indecomposable components. The subsystem $\Phi'$ is the union of the components intersecting $B_\infty$ non-trivially, so in this case $\Phi'$ is the indecomposable component $\Phi_2$ associated to the block $\{-1,4\}$. 
	
	It remains to deduce the element $w\in W_F = \tC_2 \times \tA_1$. This $w$ is the product of a signed affine permutation $w_1$ of $\{-3,-2,0,2,3\}$ and an affine permutation $w_2$ of $\{-1,4\}$. We determine $w_1$ as the element with inversion set equal to 
	\[B\cap \Phi_1^+ = \{\beta_{-3,3},\beta_{-2,2},\beta_{-2,3},\beta_{2,3} \}.\]
	This is the element $w_1$ sending $(-3,-2,0,2,3)$ to $(3,2,0,-2,-3)$. And since $\Phi_2\subseteq \Phi'$, we compute $w_2$ as the element with inversion set equal to 
	\[ \Phi_2^+\setminus B = \{ \beta_{-1,4} \}, \]
	which is the affine permutation sending $(-1,4)$ to $(4,-1)$.
	
	We deduce that $B=B(F,\Phi',w)$. We can compute that the total order associated by \Cref{thm:typeCorders} to $B(F,\Phi)$ is
	\[ \cdots\prec 5 \prec 1 \prec -4\prec -8\prec \cdots \prec -6\prec-3\prec -2 \prec 0 \prec 2 \prec 3 \prec 6  \prec \cdots \]
	\[ \cdots \prec 8\prec 4 \prec -1 \prec -5 \prec \cdots ,\]
	and the total order associated to $B$ is the result of acting by $w$ on this total order. Hence the total order associated to $B$ is
	\[ \cdots\prec 10 \prec -4 \prec 1\prec -13\prec \cdots \prec -12\prec 3\prec 2 \prec 0 \prec -2 \prec -3 \prec 12  \prec \cdots \] 
	\[ \cdots \prec 13\prec -1 \prec 4\prec -10 \prec \cdots.\]
	One can verify that $B$ consists of exactly those $\beta_{i,j}$ with $i<j$ so that $i\succ j$ in this total order.
	
\end{eg}

\begin{cor}
	We can also identify $\bic(\tC_n)$ with total orders on $\ZZ \setminus (N \ZZ)$ which are invariant under translation and negation, modulo reversing intervals of the form $r+N \ZZ$.
\end{cor}

\begin{proof}[Proof sketch]
	Given such a total order on $\ZZ$, we can always restrict it to $\ZZ \setminus (N \ZZ)$. Conversely, given such a total order on $\ZZ \setminus (N \ZZ)$, we know that $0$ must lie between $-i$ and $i$ for any  $i$, so it is clear where to insert each element of $N \ZZ$ into $\ZZ \setminus (N \ZZ)$. The only ambiguity would be if all of $N \ZZ$ was inserted into the same position, in which case we wouldn't know whether to insert it forwards or backwards, but this is exactly the case where $N \ZZ$ forms an interval for the total order on $\ZZ$.
\end{proof}

\begin{theorem}
	The poset $\bic(\tC_n)$ is a lattice.
\end{theorem}

\begin{proof}
	We define an involution $\sigma$ on $\bic(\tA_{N-1})$: If $X \in \bic(\tA_{N-1})$ corresponds to the total order $\prec$, define a new total order $\prec_{\sigma}$ by $x \prec_{\sigma} y$ if and only if $(-y) \prec (-x)$. Let $\sigma(X)$ correspond to the total order $\prec_\sigma$. It is easy to check that this operation commutes with the equivalence relation mapping $L_N$ to $\bic(\tA_{N-1})$, so $\sigma$ is an involution on $\bic(\tA_{N-1})$.
	It is also clear that $\sigma$ is a poset automorphism. Therefore, we have $\bigjoin \sigma(\cX) = \sigma \left( \bigjoin \cX \right)$ and $\bigmeet \sigma(\cX) = \sigma \left( \bigmeet \cX \right)$ for any subset $\cX$ of $\bic(\tA_{N-1})$. 
	
	From our combinatorial model for $\bic(\tC_n)$ above, we see that $\bic(\tC_n)$ is the fixed points of $\sigma$. The compatibility of $\sigma$ with the lattice operations, noted above, then implies that $\bic(\tC_n)$ is a sublattice of $\bic(\tA_{N-1})$.
\end{proof}

\subsection{The affine group \texorpdfstring{$\tB_n$}{\~B\_n}} \label{sec:Bmodel}
We continue to set $N = 2n+1$ and to impose that the variables $i$, $j$, $k$ and $\ell$ represent integers $\not\equiv 0 \bmod N$.

\begin{definition}
	The affine group $\tB_n$ is the subgroup of $\tC_n$ consisting of permutations $f : \ZZ \to \ZZ$ in $\tC_n$ such that $\#\{ x \in \ZZ :\ x \geq n+1,\ f(x) \leq n \}$ is even. 
\end{definition}
For example, $t^{\tC}_{(-1) 1}$ is in $\tB_n$
but $t^{\tC}_{n(n+1)}$ is not in $\tB_n$. 

\begin{remark} \label{rem:tB1}
	Usually, $\tB_n$ is defined for $n \geq 2$, but our definition also works for $n=1$. In this case, $\tB_1$ is the subgroup of $\tA_2$ generated by $t^{\tC}_{(-1) 1}$ and $t^{\tC}_{15}$. We note that this group is isomorphic to $\tA_1$. The isomorphism can be seen either by restricting the action to $\{ \pm 1 \} + 6 \ZZ$ or to $\{ \pm 2 \} + 6 \ZZ$, either of which $\tB_2$ permutes within itself.
\end{remark}

The Coxeter diagram for $\tB_n$, when $n\geq 2$, is:
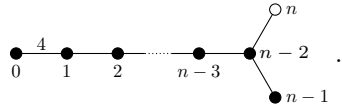
\begin{center}
	\begin{dynkinDiagram}[scale=1.5,extended,Coxeter,labels={n,n-1,,n-3,2,1,0},backwards=true,edge length = .45cm]{B}{}
		\draw node[right,font=\scriptsize] at (root 2) {$n-2$};
	\end{dynkinDiagram}.
\end{center}

The reflections of $\tB_n$ are $t^{\tC}_{ij}$ for $i<j$ with either $i \not\equiv \pm j \bmod N$ or $i+j \equiv 0 \bmod 2N$. 
The positive root corresponding to $t^{\tC}_{ij}$ for $i \not\equiv \pm j \bmod N$ is $\beta_{ij}\coloneqq \te_j - \te_i$;
the positive root corresponding to $t^{\tC}_{ij}$ for $i+j \equiv 0 \bmod 2N$ is $\beta'_{ij}\coloneqq\tfrac{1}{2} \left(\te_j - \te_i \right)$.
The simple roots are $\alpha_0\coloneqq\beta'_{-1,1}$, $\alpha_i\coloneqq\beta_{i,i+1}$ for $1\leq i\leq n-1$, and $\alpha_n\coloneqq \beta_{n-1,n+1}$. 

\begin{remark}
	We take $\tfrac{1}{2} \left(\te_j - \te_i \right)$ rather than $\te_j - \te_i$ to match the standard convention that  we should have a  ``short root''; this factor of $1/2$ will never matter. Additionally, our indexing is chosen to match \cite{Bjorner2005}, so in this section $\alpha_n$ is the ``extra'' root added to the type $B_n$ root system to get $\tB_n$. In the other sections, that role is played by $\alpha_0$.
\end{remark}

\begin{remark} \label{NotSublattice}
	The inclusion $\tB_n \into \tC_n$ is a map of groups, and of posets (using weak order), but it is not a map of lattices. 
	We can see this issue by thinking about the subgroups of $\tB_3$ and $\tC_3$ taking $\{ 1,2,3,4,5,6 \}$ to itself, which are isomorphic to $D_3$ and to $C_3=B_3$, respectively. 
	Restricting our attention to $\{ 1,2,3,4,5,6 \}$, consider the permutations $624351$ and $365214$, both of which are in $D_3$. 
	Their join in $B_3$ is $654321$, whereas their join in $D_3$ is $653421$. 
	(Note that  $653421$ is not greater than $624351$ in $B_3$, because $(3,4)$ is a $B_3$-inversion of $624351$ and not of $653421$, but $(3,4)$ is not a $D_3$-inversion, so $653421>624351$ in $D_3$.)
	The lack of an inclusion $\iota: D_3 \into B_3$ which is a map of lattices is a major obstacle to mimicking our proof of the lattice structures in $\tA$ and $\tC$ for types $\tB$ and $\tD$.
\end{remark}

Although for group theory purposes it is useful to think of $\tB_n$ as a subgroup of $\tC_n$, for lattice purposes it is often better to think of $\tB_n$ as a quotient. 
Namely, each coset in $\tC_n/\langle t^{\tC}_{n(n+1)} \rangle$ contains a unique element of $\tB_n$. 
In other words, an element of $\tB_n$ can be thought of as a permutation in $\tC_n$, modulo the operation of switching the values of $f(n + rN)$ and $f(n+1 + rN)$. 
Note that we always have $f(n + rN) + f(n+1+rN) \equiv N \bmod 2N$ and that, conversely, if $i$ and $j$ with $i+j \equiv N \bmod 2N$ occur as consecutive values $f(m)$ and $f(m+1)$, then we must have $m \equiv n \bmod N$. 
So we can also describe this as being allowed to switch consecutive elements $i$ and $j$ if $i+j \equiv N \bmod 2N$.

\begin{remark}
	This quotient map $\tC_n \to \tB_n$ is a map of posets, but it is not a map of groups (since $\langle t^{\tC}_{n(n+1)} \rangle$ is not normal) and it is not a map of lattices. 
	As in \Cref{NotSublattice}, the absence of a surjection of lattices $\pi : \tC_n \to \tB_n$  is a major obstacle to mimicking our proof of the lattice structures in $\tA$ and $\tC$ for types $\tB$ and $\tD$.
\end{remark}

We now proceed to describe $\bic(\tB_n)$. Proving that it is a lattice will occur in a future paper.

\begin{theorem}
	Let $\prec$ be a total order on $\ZZ$ which is invariant for translation (meaning $i \prec j$ implies $(i+N) \prec (j+N)$) and for negation (meaning that $i \prec j$ implies $-j \prec -i$).
	Then $\{ \te_j - \te_i : i<j,\ i \succ j,\ i \not\equiv j \bmod N,\ i+j \not \equiv N \bmod 2 N \}$ is a biclosed set of $\tB_n$, and all biclosed sets for $\tB_n$ are of this form.
	Two total orders on $\ZZ$ give the same biclosed set if and only they are the same up to (1) reversal on intervals of the form $r+N \ZZ$ and (2) interchanging neighbors $i$ and $j$ for which $i+j \equiv N \bmod 2N$
	.\end{theorem}

\begin{proof}[Proof sketch]
	Since the Coxeter groups $B_n$ and $C_n$ are the same, the faces of the $B_n$ Coxeter fan are, like those of the $C_n$ Coxeter fan,  indexed by the signed ordered set partitions.
	Given a signed set partition $(p_{-r}, \ldots, p_{-2}, p_{-1}, p_0, p_1, p_2, \ldots, p_r)$, the corresponding face subgroup is $\tB_{(|p_0|-1)/2} \times \prod_{i=1}^r \tA_{|p_i|-1}$, with the conventions that $\tB_0$ and $\tA_{1-1}$ are the trivial group and $\tB_1$, as mentioned in \Cref{rem:tB1}, is $\tA_1$.
	
	As in the $\tC$ proof, this gives us a total order on $\ZZ$, except that we need to keep track of the $\tA_{1-1}$ factors, coming from the singleton $p_i$ for $i$, and we need to deal with the fact that our first factor is $\tB_{(|p_0|-1)/2}$, not $\tC_{(|p_0|-1)/2}$. 
	As in the $\tC$ proof, the singleton $p_i$ are addressed by allowing us to reverse intervals of the form $r+N \ZZ$.
	
	We now must check that the freedom to interchange $i$ and $j$, when $i+j \equiv N \bmod 2N$, is equivalent to working with $\tB_{(|p_0|-1)/2}$ rather than $\tC_{(|p_0|-1)/2}$. 
	First, note that such $i$ and $j$ can only occur together in the same block if that block is the central block, so we only need to study what happens to our order on $p_0 + N \ZZ$.
	As discussed above, switching from $\tC_{(|p_0|-1)/2}$ to $\tB_{(|p_0|-1)/2}$ corresponds to switching consecutive values $i$ and $j$ with $i+j \equiv N \bmod 2N$. 
\end{proof}

\subsection{The affine group \texorpdfstring{$\tD_n$}{\~D\_n}} \label{sec:Dmodel}
Finally, we come to $\tD_n$. 
We continue to use our conventions that $N = 2n+1$ and $i$, $j$, $k$, $\ell$ are not $0 \bmod N$.

\begin{definition} For $n \geq 2$, the group $\tD_n$ is the subgroup of $\tC_n$ where both $\# \{ x : x > 0, f(x) < 0 \}$ and $\# \{ x : x \geq n+1 , f(x) \leq n \}$ are even.
\end{definition}

The Coxeter diagram for $\tD_n$, when $n\geq 3$, is:
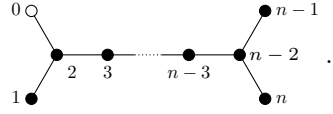
\begin{center}
	\begin{dynkinDiagram}[scale=1.5,extended,Coxeter,labels={0,1,2,3,n-3,,n-1,n},edge length = .45cm]{D}{}
		\draw node[right,font=\scriptsize] at (root 5) {$n-2$};
	\end{dynkinDiagram}.
\end{center}

When $n\geq 2$, the reflections in $\tD_n$ are $t^{\tC}_{ij}$ for $i \not \equiv \pm j \bmod N$, and the positive roots are $\beta_{ij}\coloneqq \te_j - \te_i$ for $i<j$, $i \not\equiv \pm j \bmod N$. 
In this case, we may also think of $\tD_n$ as a quotient of $\tC_n$ by $\langle t^{\tC}_{(-1)1}, t^{\tC}_{n(n+1)} \rangle$. When $n\geq 3$, the simple roots are $\alpha_0\coloneqq \beta_{-1,2}$, $\alpha_i\coloneqq\beta_{i,i+1}$ for $1\leq i \leq n-1$, and $\alpha_n\coloneqq \beta_{n-1,n+1}$.

We now discuss small values of $n$.
We note that this discussion is important, because our face subgroups will contain factors $\tD_r$ for $r \geq 1$, so we need to know what we mean by $\tD_1$, $\tD_2$ and $\tD_3$.

For $n=1$, we define $\tD_1$ to be the trivial group. 
This makes sense from a root system perspective: The roots of $\tD_1$ should be $\te_j - \te_i$ where $i$, $j \not \equiv 0 \bmod 3$ and $i \not \equiv \pm j \bmod 3$, and there are no solutions to these equations. 
It also makes sense when thinking of $\tD_1$ as a quotient of $\tC_1$: the reflections $t^{\tC}_{(-1)1}$ and $t^{\tC}_{12}$ generate $\tC_1$.
If we literally used the definition above as a subgroup of $\tC_1$, this would not work. For any $x \equiv 1 \bmod 6$, the permutation defined by $f(3n \pm 1) = 3n \pm x$ and $f(3n) = 3n$ obeys the given parity conditions, but we will not allow this.

For $n=2$, we can use the definition above. 
The root system then is $\{ \te_j - \te_i : i \not \equiv \pm j \bmod 5 \}$. We note that the roots with $(i,j) \equiv \pm (1,2) \bmod 5$ form one root system and those with $(i,j) \equiv \pm (1,3) \bmod 5$ form a second, orthogonal root system. The simple roots of the first system are $\te_2-\te_1$,$\te_6-\te_2$ and the simple roots of the second system are $\te_3-\te_1$,$\te_6-\te_3$.
We thus have $\tD_2 \cong \tA_1 \times \tA_1$ as groups and $\bic(\tD_2) \cong \bic(\tA_1) \times \bic(\tA_1)$. 
The fact that the group $\tD_2$ factors will be a major inconvenience for us, because it means that, whenever we have a factor of $\tD_2$ in our face subgroups, we get two different irreducible factors in that subgroup, so there are twice as many options for $\Phi'$ as we might guess.

Finally, we note that have $\tD_3 \cong \tA_3$, as Coxeter groups and as root systems, so $\bic(\tD_3) \cong \bic(\tA_3)$. 
This isomorphism is interesting but not important for us; we can simply study $\tD_3$ separately by the methods in this section without needing to consider this isomorphism.

To index the faces of the $D_n$ Coxeter arrangement, we use a variant of signed ordered set partitions.
We define a \newword{type $D$ signed ordered set partition} to be an ordered set partition of $\{ \pm 1, \pm 2, \ldots, \pm n \}$ into sets $p_{-r}, \ldots, p_{-1}, p_0, p_1, \ldots, p_r$ such that $p_{-i} = - p_i$ and $p_0$ is nonempty. (Note that $|p_0|$ is even now.)

We consider the $D_n$ Coxeter fan as living in $\RR^n$ with coordinates $x_1, x_2, \ldots, x_n$.
The $D_n$ hyperplanes are the hyperplanes of the form $x_i = \pm x_j$ for $i \neq j$. 
We put $x_{-i} = - x_i$. 

The type $D$ signed ordered set partition $(p_{-r}, \dots, p_{-1}, p_0, p_1, \dots, p_r)$ corresponds to the region in $\RR^n$ where
\begin{enumerate}
	\item For $a > 0$ and $p_a = \{ i_1, i_2, \ldots, i_k \}$, we have $x_{i_1} = x_{i_2} = \cdots = x_{i_k}>0$.
	\item If $|p_0| > 2$ and $p_0 = \{ \pm i_1, \pm i_2, \ldots, \pm i_k \}$, we have $x_{i_1} = x_{i_2} = \cdots = x_{i_k} = 0$. 
	\item For $-r \leq a < b \leq r$, with $i \in p_a$ and $j \in p_b$, we have $x_i < x_j$.
\end{enumerate}

\begin{eg}
	The type $D$ signed ordered set partition 
	\[(\{3 \}, \{ -2, 5 \}, \{ \pm 1, \pm 4 \}, \{ 2, -5 \}, \{ -3 \})\] 
	corresponds to $0=x_1=x_4 < x_2 = (-x_5) < (-x_3)$.
	The type $D$ signed ordered set partition $(\{3 \}, \{ -2, 5 \}, \{ -4 \}, \{ \pm 1 \}, \{ 4 \}, \{ 2, -5 \}, \{ -3 \})$ corresponds to $|x_1|<x_4 < x_2 = (-x_5) < (-x_3)$.
\end{eg}

Let $(p_{-r}, \dots, p_{-1}, p_0, p_1, \dots, p_r)$ be a type $D$ signed ordered set partition. 
The corresponding  face subgroup is $\tD_{|p_0|/2} \times \prod \tA_{|p_i|-1}$, with the conventions for $\tD_1$, $\tD_2$ and $\tD_3$ described above. 

We'd like to describe $\bic(\tD_n)$ using total orders on $\ZZ$, as we have in other types. 
Based on our other examples, we might  consider total orders on $\ZZ$ which are invariant for translation by $N$ and negation, modulo the relations of (1) reversing intervals of the form $r+N \ZZ$ and (2) interchanging adjacent values $i$ and $j$ where $i+j \equiv 0 \bmod N$. 
We can turn such a total order $\prec$ into the biclosed set $\{ \te_j - \te_i : i<j,\ i \succ j,\ i \not \equiv \pm j \bmod N \}$.

This would describe a large portion, but not all, of $\bic(\tD_n)$. 
Specifically, what we miss is those biclosed sets that correspond to a  type $D$ signed ordered set partition $(p_{-r}, \ldots, p_{-1}, p_0, p_1,\allowbreak \ldots, p_r)$ where $|p_0| = 4$ and where $\Phi'$ uses precisely one of the two factors of $\tD_2 \cong \tA_1 \times \tA_1$.

Let us see what these extra biclosed sets look like. Let $p_0 = \{ \pm i, \pm j \}$ for $1 \leq i < j \leq n$. 
The positive roots in $\Phi_{F}$ are 
\[ \{ \te_j - \te_i,  \te_{j+N} - \te_i,  \dots, \te_{i+2N} - \te_j, \te_{i+N} - \te_j \}\ \cup \ \{ \te_{N-j} - \te_i, \te_{2N-j} - \te_i, \ldots, \te_{i+N} - \te_{-j}, \te_i - \te_{-j} \} \]
These are two disjoint $\tA_1$'s.
The extra biclosed sets are those which have finite intersection with the first $\tA_1$ and cofinite intersection with the second $\tA_1$, or vice versa.

We thus spell out our, somewhat wordy, classification of biclosed sets in type $\tD$:

\begin{theorem}
	Let $\prec$ be a total order on $\ZZ$ which is invariant for translation (meaning $i \prec j$ implies $(i+N) \prec (j+N)$) and for negation (meaning that $i \prec j$ implies $-j \prec -i$).
	Then $\{ \te_j - \te_i : i<j,\ i \succ j,\ i \not\equiv \pm j \bmod N \}$ is a biclosed set of $\tD_n$.
	Two total orders on $\ZZ$ give the same biclosed set if and only they are the same up to (1) reversal on intervals of the form $r+N \ZZ$ and (2) interchanging neighbors $i$ and $j$ for which $i+j \equiv 0 \bmod N$.
	
	In addition to the biclosed sets described above, we must introduce some additional biclosed sets. 
	Let $\prec$ be a total order such that there is an interval of the form $\{ \pm i,\ \pm j \} + N \ZZ$ (with $\pm i \not\equiv \pm j \not\equiv 0 \bmod N$).
	Let $X$ be the biclosed set corresponding to this partial order. 
	Then we can get another biclosed set as $X \oplus \{ \te_b - \te_a : a<b,\ \{ a,b \} \equiv \{ i,j \} \bmod N\}$. 
	
	The biclosed sets coming from total orders, and their modifications as in the second paragraph, are the complete list of biclosed sets in $\tD_n$.
\end{theorem}

\printbibliography
\end{document}